\documentclass[10pt]{amsart}
\usepackage{amsmath,amsfonts,amssymb,amsthm,amscd,comment,euscript}
\usepackage[all]{xy}
\usepackage{graphicx}
\usepackage{mathtools}
\usepackage{enumerate} 
\usepackage[colorlinks=true]{hyperref} 
\usepackage[usenames,dvipsnames]{xcolor}
\usepackage{enumitem,multicol}
\usepackage{fancyhdr}
\usepackage[yyyymmdd,hhmmss]{datetime}

\swapnumbers
\theoremstyle{plain}
\newtheorem{lem}{Lemma}[section]

\newtheorem{cor}[lem]{Corollary}
\newtheorem{prop}[lem]{Proposition}
\newtheorem{thm}[lem]{Theorem}

\theoremstyle{definition}

\newtheorem{rem}[lem]{Remark}
\newtheorem{dfn}[lem]{Definition}

\newcommand{\w}{\wedge}
\newcommand{\ZF}{\operatorname{\mathbb{Z}F}}
\newcommand{\PP}{\mathbb{P}}

\newcommand{\Sm}{\mathrm{Sm}}
\newcommand{\A}{\mathbb{A}}
\newcommand{\Os}{\mathcal{O}}
\newcommand{\Z}{\mathbb{Z}}

\newcommand{\Gm}{\mathbb{G}_m}

\newcommand{\colim}{\operatorname{colim}}
\newcommand{\Hom}{\operatorname{Hom}}
\newcommand{\Ext}{\operatorname{Ext}}
\newcommand{\MW}{\mathrm{MW}}
\newcommand{\HZF}{\mathrm{H}_0^{\ZF}}
\newcommand{\HCor}{\mathrm{H}_0^{\MW}}
\newcommand{\SH}{\mathrm{SH}}
\newcommand{\Fr}{\operatorname{Fr}}
\newcommand{\Frn}{\Fr^{\mathcal{V}}}
\newcommand{\hFrn}{\overline{\Fr}^{\mathcal{V}}}

\newcommand{\coker}{\operatorname{coker}}
\newcommand{\Shvb}{\mathrm{Shv}_{\bullet}}

\newcommand{\Spec}{\operatorname{Spec }}
\newcommand{\GW}{\mathrm{GW}}
\newcommand{\id}{\mathrm{id}}

\newcommand{\KstarMW}{\mathrm{K}_*^{\mathrm{MW}}}
\newcommand{\KMW}{\mathrm{K}^{\mathrm{MW}}}
\newcommand{\wCor}{\widetilde{\operatorname{Cor}}}
\newcommand{\Ab}{\mathrm{ShvAb}}
\newcommand{\ShvMW}{\mathrm{ShvAb}_k^\MW}
\newcommand{\Shvfr}{\mathrm{ShvAb}_k^\mathrm{fr}}
\newcommand{\HM}{\mathrm{\Pi}_*(k)}
\newcommand{\HMfr}{\mathrm{\Pi}_*^{\mathrm{fr}}(k)}
\newcommand{\Phifr}{\Phi^{\mathrm{fr}}}
\newcommand{\Psifr}{\Psi^{\mathrm{fr}}}
\newcommand{\HMCor}{\mathrm{\Pi}_*^{\mathrm{MW}}(k)}
\newcommand{\wDM}{\widetilde{\operatorname{DM}}}

\author{Alexey Ananyevskiy}
\address{St. Petersburg Department, Steklov Math. Institute, Fontanka 27, St. Petersburg 191023 Russia, and Chebyshev Laboratory, St. Petersburg State University, 14th Line V.O., 29B, St. Petersburg 199178 Russia}
\email{alseang@gmail.com}
\author{Alexander Neshitov}
\address{Department of Mathematics, University of Southern California,
3620 S. Vermont ave, Los Angeles, CA, 90089}
\email{neshitov@usc.edu}
\title{Framed and MW-transfers for homotopy modules}



\thanks{The first author is supported by Young Russian Mathematics award, by "Native towns", a social investment program of PJSC "Gazprom Neft", and by RFBR grant 18-31-20044.}

\begin{document}

\begin{abstract}
In the paper we use the theory of framed correspondences to construct Milnor-Witt transfers on homotopy modules. As a consequence we identify the zeroth stable $\A^1$-homotopy sheaf of a smooth variety with the zeroth homology of the corresponding MW-motivic complex and prove that the hearts of the homotopy $t$-structures on the stable $\A^1$-derived category and the category of Milnor-Witt motives are equivalent. We also show that a homotopy invariant stable linear framed Nisnevich sheaf admits a unique structure of a presheaf with MW-transfers compatible with the framed structure. 

\end{abstract}
\maketitle
\section{Introduction}
The theory of framed correspondences was introduced by Voevodsky in~\cite{V} and developed by Garkusha and Panin \cite{GP,GPhi}. This theory can be seen as an enhancement of Voevodsky's theory of presheaves with transfers \cite{MVW06} adjusted to the setting of motivic homotopy theory. Presheaves with transfers have wrong-way maps (transfers) along finite surjective morphisms $Z\to X$; for framed presheaves one has transfers along finite surjective morphisms $Z\to X$ equipped with the additional data: (1) a factorization $Z\xrightarrow{i} \mathbb{A}^n\times X\xrightarrow{p} X$, where $i$ is a regular embedding and $p$ is the projection, and (2) trivialization of the normal bundle $N_i$. Every presheaf on the category of smooth varieties over a field that can be factored through the motivic stable homotopy category has a canonical structure of a presheaf with framed transfers. 

Calm\`es and Fasel introduced another enhancement of presheaves with transfers, the so-called presheaves with Milnor--Witt transfers \cite{CF}. These presheaves have wrong-way maps along the finite surjective maps equipped with a certain unramified quadratic datum (that endows the morphism with a kind of an orientation). The theory was subsequently developed in \cite{CF2,DF,FO}.

In this paper we study the relations between the theories based on framed and Milnor--Witt correspondences over an infinite perfect field $k$ of characteristic different from $2$. We construct a ``forgetful'' functor $\ZF_*(k)\to \wCor(k)$ (see Proposition~\ref{prop:functorcom}; there is another construction given in~\cite{DF}) and obtain the following comparison results.

\begin{thm}[{Theorem~\ref{hcor=hzf}}]
For a smooth variety $X$ the forgetful functor induces an isomorphism $\HZF(X)\cong \HCor(X)$ of the zeroth $\ZF$ and $\MW$-homology sheaves.
\end{thm}
\noindent Recall that $\pi_0(\Sigma_{\PP^1}^\infty X_+)_0\cong\HZF(X)$ by~\cite[Theorem~11.1 and Corollary~11.3]{GP}, so this theorem provides a reasonable description for the zeroth stable motivic homotopy sheaf of $X$. A similar answer with $X$ being a smooth projective variety was obtained by different techniques in \cite[Theorem~4.3.1]{AH} and \cite[Theorem~1.2]{Ananyevskiy}, see also \cite[Theorem~1.3]{Ananyevskiy} for the case of a smooth curve.

\begin{cor}[Corollary~\ref{cor:addmw}]
Let $M$ be a homotopy invariant stable $\ZF_*$-sheaf. Then $M$ admits a unique structure of a presheaf with MW-transfers compatible with the framed structure.
\end{cor}

\begin{thm}[{Theorem~\ref{hmodeq}}]
The forgetful functors induce equivalences
\[
\HMCor  \xrightarrow{\simeq} \HMfr \xrightarrow{\simeq} \HM
\]
between the categories of homotopy modules, framed homotopy modules and homotopy modules equipped with MW-transfers. 
\end{thm}

\noindent 
The category of homotopy modules \cite[Section~5.2]{Morelintro} can be identified with the heart of the homotopy $t$-structure on the stable $\A^1$-derived category (see, for example, \cite[Theorem~3.3.3]{AH}). In a similar way one identifies the category $\HMCor$ with the heart of the homotopy $t$-structure on the category of $\MW$-motives $\wDM(k)$. Thus we have the following corollary.

\begin{cor}[{Corollary~\ref{prop:hearts}}]
The forgetful functor $\wDM(k)\to \mathrm{D}_{\A^1}(k)$ induces an equivalence of the hearts of the homotopy $t$-structures.
\end{cor}

The paper is organized as follows. In section~\ref{hiss} we recall the main definitions of the theory of framed correspondences and list important properties of presheaves with framed transfers that follow from~\cite{GPhi}. In section~\ref{fchm} we construct framed transfers on a homotopy module and prove that every homotopy module has a unique structure of a stable $\ZF_*$-sheaf such that the suspension isomorphisms agree with framed transfers. In section~\ref{ftns} we provide a description of framed transfers via a trivialization of the conormal sheaf for the support (see Corollary~\ref{lem:frameeq}). In sections~\ref{pfmocohm} and~\ref{sect:pushgen} we use the framed structure on homotopy modules to construct the pushforward map on cohomology groups of homotopy modules. For a map $f\colon X\to Y$ of smooth varieties with a closed subset $Z\subset X$ finite over $Y$ we locally construct a framed enhancement for $f$ giving a wrong-way framed correspondence then use Cousin complex to glue the induced maps into a pushforward map from the twisted cohomology of $X$ with support on $Z$ to the cohomology of $Y$ (Definition~\ref{dfn:pushc}). The resulting pushforward maps coincide with the transfer maps constructed in~\cite[\S 3,4]{Morel}, but we need the framed interpretation to prove a special case of the base change theorem~\ref{basechange2}. In section~\ref{mwc} we recall the main definitions of the theory of Milnor--Witt correspondences introduced in~\cite{CF} and discuss the natural action of the framed correspondences on the Milnor--Witt ones. Then we use the pushforward maps introduced in section~\ref{sect:pushgen} to define an action of Milnor--Witt correspondences on the cohomology groups of homotopy modules. In Proposition~\ref{prop:functorcom} we construct a functor from the category of linear framed correspondences to the category of $\MW$-correspondences. Such functor was also constructed in~\cite{DF}, but our approach is a different one and is based on the natural action of framed correspondences on the cohomology of Milnor--Witt K-theory. In the section~\ref{mwohm} we prove that every homotopy module has a unique structure of a presheaf with Milnor--Witt transfers that agrees with the framed structure. Then we derive the two main results of the article: establish an isomorphism between the zeroth stable motivic homotopy sheaf of $X$ with the zeroth homology of the $\MW$-motive of $X$ and show that the categories of homotopy modules, homotopy modules with $\MW$-transfers and framed homotopy modules are equivalent.

Throughout the paper we employ the following notation.
\vspace{0.05in}

\begin{tabular}{l|l}
	$k$ & an infinite perfect field of $\operatorname{char} \neq 2$ \\
	$\Sm_k$ & the category of smooth quasi-projective schemes over $k$\\
	$\Shvb$ & the category of pointed Nisnevich sheaves of sets on $\Sm_k$\\
	$\Ab$ & the category of Nisnevich sheaves of abelian groups on $\Sm_k$\\ 
	$X_+$ & the pointed scheme $X\sqcup\Spec k$ considered as an object of $\Shvb$\\
	$\PP^{1}$ & the pointed sheaf represented by the pointed scheme $(\PP^1,\infty)$\\
	$T$ & the quotient $\A^1/\Gm$ considered as an object of $\Shvb$\\
	$\sigma\colon \PP^{1}\to T$ & the composition $\PP^{1}\to \PP^1/(\PP^{1}-0) \cong \A^1/\Gm = T$ \\
	$\PP^{\wedge n}, T^n$ & the $n$-fold smash power of the respective objects of $\Shvb$ \\
	$\sigma^n\colon\PP^{\w n}\to T^n$ & the $n$-fold smash power of $\sigma$ \\
	sheaf & Nisnevich sheaf
\end{tabular}
\vspace{0.05in}
\noindent

\subsection{Acknowledgements} This work was initiated during the program "Algebro-Geometric and Homotopical Methods" at the Institut Mittag-Leffler in 2017. Authors would like to thank the organizers Eric Friedlander, Lars Hesselholt, and Paul Arne \O stvaer.
\section{Homotopy invariant stable $\ZF_*$-sheaves}\label{hiss}

The notion of framed correspondences was introduced by Voevodsky in~\cite{V} and studied by Garkusha and Panin in~\cite{GP}. We briefly recall the main definitions and constructions and refer the reader to~\cite{GP} and~\cite{GPhi} for the details.

\begin{dfn}
For $X,Y\in \Sm_k$ and closed subsets $A\subset X$ and $B\subset Y$ let $\Fr(X/A,Y/(Y-B))$ be the set of equivalence classes of triples $(Z,U,g)$ where 
\begin{enumerate}
\item 
$Z\subset X$ is a closed subset such that $Z\cap A=\emptyset$, 
\item
$U$ is an \'etale neighborhood of $Z$ in $X$,
\item
$g\colon U\to Y$ is a regular morphism such that $Z=g^{-1}(B)$. 
\end{enumerate}
The triples $(Z,U,g)$ and $(Z',U',g')$ are said to be equivalent if $Z=Z'$ and there is an \'etale neighborhood $V$ of $Z$ in $U\times_X U'$ such that $g$ and $g'$ coincide on $V.$
\end{dfn}

\begin{lem}[{Voevodsky; see also \cite[Corollary~3.3]{GP}}]\label{Vlemma}
For $X,Y\in \Sm_k$ and closed subsets $A\subset X$ and $B\subset Y$ there is a canonical bijection
\[
\Hom_{\Shvb}(X/A,Y/(Y-B))\cong \Fr(X/A,Y/(Y-B)).
\]
\end{lem}

\begin{dfn}\label{def:frn} For $F,G\in\Shvb$ put
\[
\Fr_n(F,G)=\Hom_{\Shvb}(F\wedge\PP^{\wedge n},G\wedge T^n)
\]
For $F,G,H\in \Shvb$ and $a\in \Fr_n(F,G),\,b\in \Fr_m(G,H)$ denote $b\circ a\in \Fr_{n+m}(F,H)$
the composition
\[
F\wedge\PP^{\w n}\wedge\PP^{\w m}\stackrel{a\w \id}\to G\wedge T^n\wedge\PP^{\w m}\stackrel{tw_1}\to G\wedge\PP^{\w m}\wedge T^n\stackrel{b\w \id}\to H\wedge T^m\wedge T^n\stackrel{tw_2}\to H\wedge T^n\wedge T^m
\]
where $tw_1$ permutes $\PP^{\w m}$ with $T^n$ and $tw_2$ permutes $T^m$ with $T^n$.
\end{dfn}

\begin{dfn}\label{def:frnxy}
For $X,Y\in\Sm_k$ we write $\Fr_n(X,Y)$ for $\Fr_n(X_+,Y_+)$. The set $\Fr_n(X,Y)$ is pointed by the trivial morphism sending everything to the distinguished point. Lemma~\ref{Vlemma} in this particular case identifies $\Fr_n(X,Y)$ with the set of triples $(Z,U,g)$ where $Z$ is a closed subset of $X\times (\mathbb{P}^1)^{\times n}$ that does not intersect the complement of $\A^n_X$. Thus $Z$ is simultaneously affine and projective over $X$, hence finite. Writing the map $g\colon U\to\A^n_Y$ as a pair $g=(\phi,f), \phi\colon U\to\A^n_k, f\colon U\to Y,$ we get the following description: the set $\Fr_n(X,Y)$ equals to the set of equivalence classes of quadruples $(Z,U,\phi,f)$ where 
\begin{enumerate}
\item 
$Z$ is a closed subset of $\A^n_X$ finite over $X$,
\item
$U$ is an \'etale neighborhood of $Z$ in $\A^n_X$,
\item
$\phi\colon U\to\A^n_k$ is a regular morphism such that $\phi^{-1}(0)=Z$,
\item
$f\colon U\to Y$ is a regular morphism.
\end{enumerate}
The quadruples $(Z,U,\phi,f)$ and $(Z',U',\phi',f')$ are equivalent if $Z=Z'$ and there is an \'etale neighborhood $V$ of $Z$ in $U\times_X U'$ such that $\phi$ and $f$ coincide on $V$ with $\phi'$ and $f'$ respectively.
\end{dfn}

\begin{dfn}[{\cite[Definition~2.3]{GP}}] The category $\Fr_+(k)$ has the objects those of $\Sm_k$, the hom sets given by 
\[
\Hom_{\Fr_+(k)}(X,Y)=\bigvee\limits_{n\geqslant 0} \Fr_n(X,Y)
\]
and the composition given by Definition~\ref{def:frn}. There is an obvious functor $\Sm_k\to\Fr_+(k)$ identical on objects and tautological on the morphisms
\[
\Hom_{\Sm_k}(X,Y)\to \Hom_{\Shvb}(X_+,Y_+) = \Fr_0(X,Y).
\]
A presheaf on the category $\Fr_+(k)$ is called a framed presheaf. For $X\in\Sm_k$ we denote $\Fr_+(X)$ the framed presheaf given by $U\mapsto \Fr_+(U,X)$.
\end{dfn}

\begin{dfn}[{\cite[Definition~2.4]{GP}}] \label{def:frextprod}
For $F,F',G,G'\in\Shvb$ and $a\in \Fr_n(F,G),\, b\in  \Fr_m(F',G')$ the external product $a\boxtimes b\in \Fr_{n+m}(F\w F',G\w G')$ is given by the composition
\[
F\w F'\w \PP^{\w n}\w \PP^{\w m}\stackrel{tw_1}\to F\w \PP^{\w n}\w F'\w \PP^{\w m}\stackrel{a\w b}\to G\w T^n\w G'\w T^m\stackrel{tw_2}\to G\w G'\w T^n\w T^m
\]
where $tw_1$ twists $\PP^{\w n}$ with $F'$ and $tw_2$ twists $T^n$ with $G'$.
\end{dfn}

\begin{dfn} \label{def:susp}
The external product with $\sigma\colon \PP^1\to T$ considered as an element of $\Fr_1(k,k)$ gives rise to the suspension map
\[
\Sigma\colon\Fr_n(F,G)\to\Fr_{n+1}(F,G),\quad
a\mapsto  a\boxtimes\sigma.
\]
For $F\in\Shvb$ denote $\sigma_F$ the element $\Sigma(\id_F)\in\Fr_1(F,F)$. Then for $a\in\Fr_n(F,G)$ one has 
\[
\Sigma (a)=\sigma_G\circ a.
\]
\end{dfn}

\begin{dfn}[{\cite[Definition~8.3]{GP},~\cite[\S~2]{GPhi}}]
For $X,Y\in\Sm_k$ put
\[
\ZF_n(X,Y)=\mathbb{Z}[\Fr_n(X,Y)]/A
\] 
with $A$ being the subgroup generated by the formal linear combinations 
\[
(Z_1\sqcup Z_2,U,\phi,f) - (Z_1,U-Z_2,\phi,f) -(Z_2,U-Z_1,\phi,f).
\]
The category $\ZF_*(k)$ has the objects those of $\Sm_k$, hom sets given by 
\[
\Hom_{\ZF_*(k)}(X,Y)=\bigoplus_{n\geqslant 0}\ZF_n(X,Y)
\]
and the composition induced by the composition in $\Fr_+(k)$. Note that $\ZF_*(k)$ is an additive category with the direct sum given by $X\sqcup Y$ and that there are obvious functors $\Sm_k\to\Fr_+(k)\to\ZF_*(k)$. An additive contravariant functor $\ZF_*(k)\to\mathrm{Ab}$ is called a $\ZF_*$-presheaf (or a linear framed presheaf). A $\ZF_*$-presheaf restricting to a sheaf on $\Sm_k$ is called a $\ZF_*$-sheaf. We say that a presheaf of abelian groups $F\colon \Sm_k^{op}\to \mathrm{Ab}$ admits a structure of a $\ZF_*$-presheaf if it can be extended to a $\ZF_*$-presheaf,
\[
\xymatrix{
	\Sm_k^{op} \ar[r] \ar[dr]_{F} & \ZF_*(k)^{op} \ar[d] \\
	&  \mathrm{Ab}
}
\]
\end{dfn}

\begin{dfn}
For $X\in\Sm_k$ denote $\ZF_n(X)$ the presheaf given by
\[
U\mapsto \ZF_n(U,X).
\]
The suspension map of Definition~\ref{def:susp} descends to a morphism of presheaves 
\[
\Sigma\colon\ZF_n(X)\to\ZF_{n+1}(X)
\]
and gives rise to a $\ZF_*$-presheaf
\[
\ZF(X)=\colim(\ZF_0(X)\xrightarrow{\Sigma}\ZF_1(X)\xrightarrow{\Sigma}\ldots \xrightarrow{\Sigma} \ZF_n(X)\xrightarrow{\Sigma}\ldots).
\]
\end{dfn}

\begin{lem}[{cf. \cite[the discussion above Theorem~1.1]{GPhi}}] \label{lem:additivity}
A framed presheaf of abelian groups $F$ admits a structure of a $\ZF_*$-presheaf if and only if the natural map $F(X\sqcup Y)\to F(X)\oplus F(Y)$ is an isomorphism for all $X,Y\in\Sm_k$.
\end{lem} 
\begin{proof}
A $\ZF_*$-presheaf is an additive functor, so one implication is clear.
	
Suppose that $F$ is a framed presheaf and $F(X\sqcup Y)\to F(X)\oplus F(Y)$ is an isomorphism for every $X,Y\in\Sm_k$. For $Y\in\Sm_k$ let 
\[
p,p_1,p_2\in\Fr_0(Y\sqcup Y,Y)=\Hom_{\Shvb}(Y_+\vee Y_+,Y_+)
\]
be the following morphisms:
\begin{itemize}
\item
$p$ is the identity on both copies of $Y$,
\item
$p_1$ is the identity on the first copy of $Y$ and sends the second copy of $Y$ to the distinguished point,
\item
$p_2$ is the identity on the second copy of $Y$ and sends the first copy of $Y$ to the distinguished point. 
\end{itemize}
By the assumption we have 
\[
p^*=p_1^*+p_2^*\colon F(Y)\to F(Y\sqcup Y).
\]

Consider $c=(Z_1\sqcup Z_2,U,\phi,f)\in\Fr_n(X,Y)$. Then for
\[
c'=(Z_1\sqcup Z_2,(U-Z_2)\sqcup (U-Z_1),\phi\sqcup\phi,f\sqcup f)\in\Fr_n(X,Y\sqcup Y)
\]
we have
\[
p\circ c'=(Z_1\sqcup Z_2,U,\phi,f),\quad
p_1\circ c'=(Z_1,U-Z_2,\phi,f),\quad p_2\circ c'=(Z_2,U-Z_1,\phi,f)
\]
in $\Fr_n(X,Y)$. Put $c_1=p_1\circ c'$, $c_2=p_2\circ c'$. By the above we have
\[
c^*= (c')^*\circ p^*= (c')^*\circ (p^*_1+p_2^*)=c_1^*+c_2^*,
\]
thus $F$ descends to a functor on $\ZF_*(k)$.
\end{proof}

\begin{dfn}[{\cite[\S 1]{GPhi}}] A $\ZF_*$-presheaf $M$ is called 
\begin{itemize}
\item stable if $\sigma_X^*\colon M(X)\to M(X)$ equals to the identity map on $M(X)$,
\item homotopy invariant if for every $X\in\Sm_k$ the projection $p\colon \A^1\times X\to X$ induces an isomorphism $p^*\colon M(X)\to M(X\times\A^1)$.
\end{itemize}
\end{dfn}

\begin{dfn} For $X\in\Sm_k$ denote $\HZF(X)$ the sheaf associated with the presheaf 
\[
U\mapsto \coker\left[\ZF(U\times\A^1,X)\xrightarrow{i_0^*-i_1^*}\ZF(U,X)\right].
\]
Here $i_0,i_1\colon U\to U\times \A^1$ are given by $i_0(u)=(u,0)$ and $i_1(u)=(u,1)$ respectively.
For $n\geqslant 0$ put 
\[
\HZF(X\times\Gm^{\w n}) = \coker \left[ \bigoplus_{i=1}^n\HZF(X\times (\Gm)^{\times (n-1)})\xrightarrow{\sum f_i}\HZF(X\times(\Gm)^{\times n})\right]
\]
with $f_i$ induced by the embedding $\Gm^{n-1}\to\Gm^n$ inserting $1$ at the $i$-th place.
\end{dfn}

\begin{lem}\label{barmap}
For $X\in \Sm_k$ the sheaf $\HZF(X)$ admits a canonical structure of a homotopy invariant stable $\ZF_*$-presheaf.
\end{lem}
\begin{proof}
The presheaf 
\[
U\mapsto \coker\left[\ZF(U\times\A^1,X)\xrightarrow{i_0^*-i_1^*}\ZF(U,X)\right]
\] 
admits a canonical structure of a $\ZF_*$-presheaf being the cokernel of a morphism of $\ZF_*$-presheaves. It is immediate to see that this presheaf is homotopy invariant and~\cite[the proof of Remark~2.3]{Neshitov} yields that it is stable. Thus \cite[Theorem~2.1]{GPhi} gives the claim.
\end{proof}

\begin{lem} Let $M$ be a homotopy invariant stable $\ZF_*$-sheaf. Then for $X\in\Sm_k$ and $p\geqslant 0$ the canonical morphism 
\[
H^p_{Zar}(X,M)\to H^p_{Nis}(X,M)
\] 
is an isomorphism.
\end{lem}
\begin{proof}
Consider the Leray spectral sequence
\[
H^p_{Zar}(X,\mathcal{H}^q_{Nis}(M))\Rightarrow H^{p+q}_{Nis}(X,M)
\]
where $\mathcal{H}^q_{Nis}(M)$ is the Zariski sheafification of the presheaf $U\mapsto H^q_{Nis}(X,M)$. The homomorphism 
\[
H^q_{Nis}(\Spec \Os_{X,x},M)\to H^q_{Nis}(\Spec k(X),M)
\]
is injective by~\cite[Theorems~2.15(3), 16.10 and~16.11]{GPhi}. Thus for $q>0$ we have  
\[
H^q_{Nis}(\Spec \Os_{X,x},M)=0
\]
whence $\mathcal{H}^q_{Nis}(M)=0$.  The spectral sequence degenerates yielding the claim.
\end{proof}

\begin{dfn}
For a presheaf $M$ let $M_{-1}$ be the presheaf given by 
\[
X\mapsto \ker \left[M(\Gm\times X)\xrightarrow{i_1^*}M(X)\right]
\]
for the embedding $i_1\colon X\to\Gm\times X$, $i_1(x)=(1,x)$. It is straightforward to see that if $M$ is a homotopy invariant stable $\ZF_*$-sheaf then $M_{-1}$ is also homotopy invariant stable $\ZF_*$-sheaf.
\end{dfn}

\begin{lem}\label{minusexact}
The functor $M\mapsto M_{-1}$ is exact on the category of homotopy invariant stable $\ZF_*$-sheaves.
\end{lem}
\begin{proof}
It is clear that the functor is left exact. 

Suppose that $M\to N$ is an epimorphism of homotopy invariant stable $\ZF_*$-sheaves. The projection $p\colon \Gm\times X\to X$ gives a splitting
\[
M(\Gm\times X)\cong M_{-1}(X)\oplus M(X)
\]
functorial in $X$, thus it is sufficient to show that $M(\Gm\times U)\to N(\Gm\times U)$ is surjective for a local $U$.

The cokernel presheaf $Q=\coker(M\to N)$ is a homotopy invariant stable $\ZF_*$-presheaf as well as the presheaf $Q(\Gm\times -)$. By \cite[Theorem~2.15(3)]{GPhi} the pullback 
\[
Q(\Gm\times U) \to Q(\Gm\times \Spec k(U))
\]
is injective. Applying \cite[Theorem~2.15(1)]{GPhi} to the presheaf $Q$ considered as a $\ZF_*(k(U))$-presheaf we obtain that the pullback
\[
Q(\Gm\times k(U)) \to Q(\Spec k(\Gm\times U))
\]
is injective whence the pullback
\[
Q(\Gm\times U)\to Q(\Spec k(\Gm\times U))
\]
is injective as well. By the assumption $ Q(\Spec k(\Gm\times U))=0$ whence the claim.
\end{proof}

\section{Framed structure for homotopy modules}\label{fchm}

\begin{dfn} Let $\HM$ denote the category of homotopy modules, i.e. pairs $(M_*,\phi_*)$ where $M_*$ is a $\Z$-graded strictly homotopy invariant sheaf and $\phi_i\colon M_i\to (M_{i+1})_{-1}$ are isomorphisms. We usually shorten the notation and refer to a homotopy module $(M_*,\phi_*)$ as $M_*$.
\end{dfn}

\begin{rem} By~\cite[Theorem~5.2.6]{Morelintro} there is a pair of functors
\[\pi_0(-)_*\colon\SH(k)\leftrightarrows \HM\colon H\]
that provides an equivalence between $\HM$ and the heart of the homotopy t-structure on $\SH(k)$. Here $\pi_0(E)_*$ is the graded sheaf associated to the family of presheaves
\[
\{U\mapsto \Hom_{\SH(k)}(\Sigma^{\infty}_{\PP^1}U_+,E\w\Gm^{\w i})\}_{i\in\Z}
\]
and $HM_*$ is the Eilenberg-Maclane spectrum of the homotopy module $M_*$ \cite[the discussion above Theorem~5.2.6]{Morelintro}.
The main result of this section claims that every homotopy module $M_*$ carries a unique structure of a homotopy invariant stable $\ZF_*$-sheaf such that the isomorphisms $\phi_i\colon M_i\to (M_{i+1})_{-1}$ are compatible with the framed transfers.
\end{rem}

\begin{dfn}
For $X,Y\in \Sm_k$ a framed correspondence
\[
a\in \Fr_n(X,Y)=\Hom_{\Shvb}(X_+\w\PP^{\w n}, Y_+\w T^n)
\]
gives rise to a morphism
\[
a_*=\Sigma^{-n}_{\PP^1}((\Sigma^{\infty}_{\PP^1} (\id_{Y_+}\wedge \sigma^n))^{-1}\circ \Sigma^{\infty}_{\PP^1}a)\colon \Sigma^{\infty}_{\PP^{1}}X_+\to \Sigma^{\infty}_{\PP^{1}}Y_+
\]
in the stable motivic homotopy category $\SH(k)$,
\[
\xymatrix{
\Sigma^{\infty}_{\PP^{1}}(X_+\wedge \PP^{\w n}) \ar[r]^{\Sigma^{\infty}_{\PP^1}a} \ar[dr]_{\Sigma^{n}_{\PP^1}a_*} & \Sigma^{\infty}_{\PP^{1}}(Y_+\wedge T^{n}) \\
& \Sigma^{\infty}_{\PP^{1}}(Y_+\wedge \PP^{\w n}) \ar[u]_{\Sigma^{\infty}_{\PP^{1}}(\id_{Y_+}\wedge \sigma^n)}
}
\]
Here $\id_{Y_+}\wedge \sigma^n$ is the isomorphism induced by the canonical $\A^1$-equivalence $\sigma\colon \PP^1\to T$.
\end{dfn}

\begin{lem}\label{frhmfunctor} The assignment $X\mapsto\Sigma^{\infty}_{\PP^{1}} X_+$, $a\mapsto a_*$ defines a functor
\[r\colon\ZF_*(k)\to\SH(k)\]
that sends $\sigma_X$ to $\id_{\Sigma^\infty_{\PP^1} X_+}$.
\end{lem}
\begin{proof}
It is straightforward to see that the assignment gives rise to a functor $\tilde{r}\colon \Fr_+(k)\to\SH(k)$ with $\tilde{r}(\sigma_X)=\id_{\Sigma^\infty_{\PP^1} X_+}$. Moreover, $\tilde{r}$ satisfies 
\[
\tilde{r}(X\sqcup Y)=\tilde{r}(X)\oplus \tilde{r}(Y)
\]
for all $X,Y\in\Sm_k$. The claim follows by the same reasoning as in Lemma~\ref{lem:additivity}.
\end{proof}

\begin{dfn} \label{dfn:phifr} Let $\HMfr$ denote the category of framed homotopy modules, i.e. the category of pairs $(M_*,\phi_*)$ where $M_*$ is a $\Z$-graded homotopy invariant stable $\ZF_*$-sheaf and $\phi_i\colon M_i\to (M_{i+1})_{-1}$ are isomorphisms of $\ZF_*$-sheaves. We usually shorten the notation and refer to a framed homotopy module $(M_*,\phi_*)$ as $M_*$. Recall that a homotopy invariant stable $\ZF_*$-sheaf is strictly homotopy invariant by \cite[Theorems~16.10 and~16.11]{GPhi}. Forgetting about the framed structure we obtain a forgetful functor 
\[
\Phifr\colon \HMfr\to \HM.
\]
\end{dfn}

\begin{lem}\label{fronhmodcan}
The functor $\pi_0(-)_*\colon \SH(k)\to \HM$ factors through the category $\HMfr$ making the following diagram commute.
\[
\xymatrix{
\SH(k) \ar[d] \ar[dr]^{\pi_0(-)_*} & \\
\HMfr \ar[r]_{\Phifr} & \HM
}
\]	
\end{lem}

\begin{proof}
Lemma~\ref{frhmfunctor} yields that for $E\in\SH(k)$ and every $i\in\Z$ the presheaf 
\[
U\mapsto \Hom_{\SH(k)}(\Sigma^{\infty}_{\PP^1}U_+,E\w\Gm^{\w i})
\]
is a stable $\ZF_*$-presheaf which is obviously homotopy invariant. The associated sheaves $\pi_0(E)_i$ are homotopy invariant stable $\ZF_*$-sheaves by~\cite[Theorem~2.1]{GPhi}. For every morphism $E\to E'$ in $\SH(k)$ the induced map $\pi_0(E)_i\to\pi_0(E')_i$ is a morphism of $\ZF_*$-sheaves, thus the isomorphisms $\pi_0(E)_i\to (\pi_0(E)_{i+1})_{-1}$ associated to the isomorphisms $E\to \underline{\Hom}(\Sigma^\infty_{\PP^1}\Gm^{\w 1},E\w\Gm^{\w 1})$ respect framed structure.
\end{proof}

\begin{dfn}\label{dfn:frtransferication}
Let $\pi^{\mathrm{fr}}_0(-)_*\colon \SH(k)\to \HMfr $ be the functor given by Lemma~\ref{fronhmodcan}. Then the composition with the Eilenberg--Maclane functor gives rise to a functor
\[
\Psifr=\pi^{\mathrm{fr}}_0(-)_* \circ H\colon \HM \to \HMfr.
\]
\end{dfn}

\begin{dfn}\label{dfn:hzf}
For $X\in\Sm_k$ let $\HZF(X)_*$ be the $\Z$-graded $\ZF_*$-sheaf given by
\[
\HZF(X)_i=
\left[
\begin{array}{ll}
\HZF(X\times\Gm^{\w i}), & i\geqslant 0, \\ 
\left(\HZF(X)\right)_{-|i|}, & i<0.
\end{array}
\right.
\]
Then $\HZF(X)_*$ is a framed homotopy module by \cite[Theorem~C]{AGP}.
\end{dfn}

\begin{lem}\label{onhzfsame}
There is a canonical isomorphism of framed homotopy modules 
\[
\HZF(X)_*\cong \Psifr  \Phifr \HZF(X)_*.
\]
\end{lem}
\begin{proof}
For every $i\in \Z$ the functor $\ZF_*(k)\to\SH(k)$ of Lemma~\ref{frhmfunctor} gives rise to a morphism of $\ZF_*$-presheaves 
\[
\ZF_*(X\times\Gm^{\w i})\to \Hom_{\SH(k)}(-,\Sigma^{\infty}_{\PP^1}X_+\w\Gm^{\w i})
\]
that descends to an isomorphism of framed sheaves
\[
\HZF(X)_i\xrightarrow{\simeq} \pi^{\mathrm{fr}}_{0}(\Sigma^{\infty}_{\PP^1}X_+)_{i}
\]
by \cite[Theorem~11.1 and Corollary~11.3]{GP}. The canonical map 
\[
\Sigma^{\infty}_{\PP^1}X_+ \to H\pi_{0}(\Sigma^{\infty}_{\PP^1}X_+)_{*}
\]
induces an isomorphism 
\[
\pi^{\mathrm{fr}}_{0}(\Sigma^{\infty}_{\PP^1}X_+)_i\to \pi^{\mathrm{fr}}_{0}(H\pi_{0}(\Sigma^{\infty}_{\PP^1}X_+)_{*})_i
\]
whence
\[
\HZF(X)_i\cong \pi^{\mathrm{fr}}_{0}(\Sigma^{\infty}_{\PP^1}X_+)_{i}\cong  \Psifr  \Phifr \pi^{\mathrm{fr}}_{0}(\Sigma^{\infty}_{\PP^1}X_+)_{i} \cong \Psifr  \Phifr\HZF(X)_i. \qedhere
\]
\end{proof}

\begin{lem}\label{frhmsur} For a framed homotopy module $M_*$ there exists is a surjection of framed homotopy modules $\bigoplus_{j\in J}\HZF(X_j)_{*+n_j}\to M_*$ for some $J$, $X_j$ and $n_j$.
\end{lem}
\begin{proof}
Every $a\in M_i(X)$ defines a morphism of $\ZF_*$-presheaves $\ZF_*(X)\to M_i$. Since $M_i$ is homotopy invariant and stable, the morphism factors through $\HZF(X)$ yielding a morphism $\rho_a\colon \HZF(X)\to M_i$. The isomorphism 
\[
\phi_i\colon M_i(X)\cong  (M_{i+1})_{-1}(X) =M_{i+1}(X\times\Gm^{\w 1})
\]
combined with the reasoning above gives rise to a morphism  of $\ZF_*$-sheaves
\[
\rho_{\phi_i(a)}\colon \HZF(X)_{1}\to M_{i+1}.
\]
Iterating we obtain compatible morphisms of $\ZF_*$-sheaves
\[
\HZF(X)_{m}\to M_{i+m},
\]
i.e. a morphism of framed homotopy modules $\HZF(X)_{*-i}\to M_*$. The direct sum of such morphisms for all $a\in M_*(X)$ and all (isomorphism classes) of $X$ gives the claim.
\end{proof}

\begin{prop}\label{unfrstr} The functors 
\[
\Phifr\colon \HMfr \rightleftarrows \HM\colon \Psifr
\]
are inverse equivalences of categories.
\end{prop}
\begin{proof}
The isomorphism $\Phifr\circ \Psifr\cong \id$ is immediate.

For a framed homotopy module $M_*$ Lemma~\ref{frhmsur} provides a resolution
\[
\hdots \to \bigoplus_{j\in J_2}\HZF(X_j)_{*+n_j}\to \bigoplus_{j\in J_1}\HZF(X_j)_{*+n_j}\to \bigoplus_{j\in J_0}\HZF(X_j)_{*+n_j}\to M_*.
\]
Lemma~\ref{onhzfsame} shows that $\Psifr\circ \Phifr\cong \id$ for the resolution whence the claim.
\end{proof}

\begin{dfn} \label{dfn:cohext} Let $M_*$ be a homotopy module. For a sheaf of abelian groups $F$ put
\[
\Ext^n(F,M_*)=\Hom_{\mathrm{D}(\Ab)}(F,M_*[n])
\]
with $M_*=\bigoplus_{i\in \Z}M_i$ considered as a complex concentrated in the zeroth degree. Then for a closed subset $S\subset X\in\Sm_k$ we have the following isomorphisms:
\begin{multline*}
H^i_S(X,M_*)\cong H^i_{S\times\Gm^{\w n}}(X\times\Gm^{\w n},M_{*+n}) \cong \\ 
\cong \Hom_{\mathrm{D}(\Ab)}(\Z[(X/X-S)\wedge\Gm^{\w n}],M_{*+n}[i])\cong \\
\cong\Ext^{n+i}(\Z[(X/X-S)\w T^n],M_{*+n})\cong \\
\cong \Ext^{n+i}(\Z[(X/X-S)\w\PP^{\w n}],M_{*+n}).
\end{multline*}
The first isomorphism is given by the proof of Lemma~\ref{minusexact} combined with Proposition~\ref{unfrstr}, the second one is well-known and the rest are induced by the canonical $\A^1$-equivalences $\PP^{\w n}\sim T^n\sim \Gm^{\w n}\wedge S^n$.
\end{dfn}

\begin{lem}\label{fronhmod} For a homotopy module $M_*$ and $a\in\Fr_n(X,Y)$ the restriction $a^*\colon (\Psifr M_*)(Y)\to (\Psifr M_*)(X)$ equals to the composition
\[
M_*(Y)\cong \Ext^n(\Z[Y\w T^n],M_{*+n})\to \Ext^n(\Z[X\w \PP^{\w n}],M_{*+n})\cong M_*(X).
\]
\end{lem}
\begin{proof}
It is straightforward to see that the above rule endows $M_*$ with the structure of a framed homotopy module. The equivalence $\Psifr\circ \Phifr \cong \id$ from Proposition~\ref{unfrstr} yields that such structure is unique.
\end{proof}

\begin{dfn}\label{frcohaction}
Similar to Definition~\ref{def:frnxy}, for $X,Y\in\Sm_k$ an element $a\in \Fr_n(X_+,Y_+\wedge T^l)$ is defined by a quadruple $(Z,U,\phi,f)$ where
\begin{enumerate}
	\item 
	$Z$ is a closed subset of $\A^n_X$ finite over $X$,
	\item
	$U$ is an \'etale neighborhood of $Z$ in $\A^n_X$,
	\item
	$\phi\colon U\to\A^{l+n}_k$ is a regular morphism such that $\phi^{-1}(0)=Z$,
	\item
	$f\colon U\to Y$ is a regular morphism.
\end{enumerate}
For a closed subset $S\subset Y$ put $a^{-1}(S)=\pi(f^{-1}(S)\cap Z)$ with $\pi\colon U\to X$ being the projection map. Then $a$ takes $(X-a^{-1}(S))_+\w\PP^{\w n}$ to $(Y-S)_+\w T^{n+l}$ and induces a morphism of sheaves
\[
(X/X-a^{-1}(S))\w \PP^{\w n} \to (Y/Y-S)\w T^{n+l}.
\]
For a homotopy module $M_*$ the morphism yields a homomorphism
\[
\Ext^{n+i}(\Z[(Y/Y-S)\w T^{n+l}],M_{*+n})\to\Ext^{n+i}(\Z[(X/X-a^{-1}(S))\w \PP^{\w n}],M_{*+n}).
\]
Applying the isomorphisms from Definition~\ref{dfn:cohext} we obtain a map
\[
a^*\colon H^{i-l}_{S}(Y,M_{*+n-l})\to H^{i}_{a^{-1}(S)}(X,M_{*+n}).
\]
\end{dfn}

\begin{dfn}\label{KMWmodule}
Let $M$ be a framed presheaf. The pairing 
\[
M(U\times\Gm^{\times n})\times\Fr_+(U,\Gm^{\times n})\to M(U)
\]
given by $(m,a)\mapsto \Delta^*(\id_U\boxtimes a)^{*}(m)$ with $\Delta\colon U\to U\times U$ being the diagonal map gives rise to a pairing of framed presheaves
\[
M_{-n}\times\Fr_+(\Gm^{\wedge n})\to M.
\] 
If $M$ is a homotopy invariant stable $\ZF_*$-sheaf this pairing descends to a morphism of $\ZF_*$-sheaves
\[
M_{-n}\times\HZF(\Gm^{\wedge n})\to M.
\]
Let $M_*$ be a homotopy module. Endow it with the framed structure via the functor $\Psifr$.  In view of the identification $\HZF(\Gm^{\wedge n})\cong\KMW_n$ given by \cite[Corollary~11.3]{GP} combined with \cite[Theorem~6.4.1]{Morelintro} the pairing endows $M_*$ with a $\KMW_*$-module structure $M_{*}\times \KMW_*\to M_*$
yielding a $\KMW_*$-module structure
\[
H^i_S(Y,M_*)\times \KMW_* \to H^i_S(Y,M_*)
\]
for a closed subset $S\subset Y\in \Sm_k$.
Moreover, for $Y_1,Y_2\in \Sm_k$ and closed subsets $S_1\subset Y_1$, $S_2\subset Y_2$ we get external product
\[
H^i_{S_1}(Y_1,M_*)\times H^j_{S_2}(Y_2,\KMW_*) \xrightarrow{\times}  H^{i+j}_{S_1\times S_2}(Y_1\times Y_2,M_{*})
\]
induced by the external product in cohomology.
\end{dfn}

\begin{lem}\label{frcomextprod}
For a framed homotopy module $M_*$, smooth varieties $X_1,Y_1,X_2,Y_2\in \Sm_k$, closed subsets $S_1\subset Y_1, S_2\subset Y_2$ and $a_1\in \Fr_+(X_1,Y_1)$, $a_2\in \Fr_+(X_2,Y_2)$ the following diagram commutes.
\[
\xymatrix{
H^i_{S_1}(Y_1,M_*)\times H^j_{S_2}(Y_2,\KstarMW)\ar[r]^(0.55)\times \ar[d]^{a^*_1\times a_2^*} & H^{i+j}_{S_1\times S_2}(Y_1\times Y_2,M_{*})\ar[d]^{(a_1\boxtimes a_2)^*}\\
H^i_{\widetilde{S}_1}(X_1,M_*)\times H^j_{\widetilde{S}_2}(X_2,\KstarMW)\ar[r]^(0.55)\times & H^{i+j}_{\widetilde{S}_1\times \widetilde{S}_2}(X_1\times X_2,M_{*})
}
\]
Here $\widetilde{S}_1=a_1^{-1}(S_1)$, $\widetilde{S}_2=a_2^{-1}(S_2)$ in the notation of Definition~\ref{frcohaction}.
\end{lem}
\begin{proof}
Straightforward, since all the maps $a_1^*,a_2^*$ and $(a_1\boxtimes a_2)^*$ are induced by the respective morphisms of sheaves.
\end{proof}

\begin{lem} \label{lem:GWviacor}
Let $M_*$ be a homotopy module, $S\subset Y\in \Sm_k$ be a closed subset and $a\in \Fr_n(Y,k)$. Then
\[
((\id_Y\boxtimes a)\circ \Delta_Y)^*(\alpha) = \alpha \cdot \overline{a}
\]
for every $\alpha\in H^i_S(Y,M_*)$. Here
\begin{itemize}
	\item 
	$\Delta_Y\colon Y\to Y\times Y$ is the diagonal morphism,
	\item
	$((\id_Y\boxtimes a)\circ \Delta_Y)^*$ is given by Definition~\ref{frcohaction},
	\item
	$\overline{a}\in \KMW_0(Y)$ is the image of $a$ under the composition
	\[
	\Fr_n(Y,k)\to \HZF(k)(Y) \cong \KMW_0(Y),
	\]
	\item
	$\alpha \cdot \overline{a}$ is given by the module structure of Definition~\ref{KMWmodule}.
\end{itemize}
\end{lem}
\begin{proof}
	Both sides of the equality are given by the composition $\Xi=\alpha \circ (\id_{Y/Y-S} \wedge a)\circ (\Delta_Y\wedge \id)$,
	\[
	\Xi\colon (Y/Y-S)\wedge \PP^{\w n} \to (Y/Y-S)\wedge Y_+ \wedge \PP^{\w n} \to (Y/Y-S)\wedge T^n \to M_{*+n}[n+i]. \qedhere
	\]
\end{proof}

\begin{lem}\label{matrixdet} For $Y\in \Sm_k$ and $A\in\mathrm{GL}_{n}(Y)$ consider the correspondence 
	\[
	c_A=(Y\times\{0\}, \A^n_Y, (x_1,\hdots,x_n)\cdot A,\pi_Y)\in\Fr_n(Y,Y)
	\]
	with $x_i$ being the coordinates on $\A^n$ and $\pi_Y\colon \A^n_Y\to Y$ being the projection. Then for a framed homotopy module $M_*$ and a closed subset $S\subset Y$ the induced map
	\[
	c_A^*\colon H^i_S(Y,M_*)\to H^i_S(Y,M_*)
	\]
	coincides with the multiplication by $\langle\det(A)\rangle \in \underline{\GW}(Y)$.
\end{lem}
\begin{proof}
	We have
	\[
	c_A=(\id_Y \boxtimes a)\circ \Delta_Y
	\]
	for $a=(Y\times\{0\}, \A^n_Y, (x_1,\hdots,x_n)\cdot A,\pi)\in \Fr_n(Y,k)$ with $\pi\colon \A^n_Y\to \Spec k$ being the projection.
	Lemma~\ref{lem:GWviacor} yields that $c_A^*$ coincides with the multiplication by the image of $a$ in 
	\[
	\HZF(k)(Y)\cong \KMW_0(Y)=\underline{\GW}(Y).
	\]
	Over a local ring every matrix of determinant $1$ is elementary \cite[III, Lemma 1.4]{Weibel}, whence for every $y\in Y$ we have $A=((\det A)\oplus I_{n-1})\cdot B_y$ for the unit matrix $I_{n-1}$ and some elementary matrix $B_y\in \mathrm{GL}(\Os_{Y,y})$. Since $B_y$ is $\A^1$-homotopy equivalent to the identity matrix then the image of $a$ in $\underline{\GW}(Y)$ coincides with $\langle\det(A)\rangle$.
\end{proof}

\section{Framing as a trivialization of the conormal sheaf}\label{ftns}

\begin{dfn}
For $X,Y\in\Sm_k$ let $\Frn_n(X,Y\wedge T^l)$ be the set of triples $c=(Z,\phi,g)$ where 
\begin{itemize}
\item $Z$ is a closed (not necessarily reduced) subscheme  of $\A^n_X$ finite over $X$,
\item $\phi=(\phi_1,\ldots,\phi_{n+l})$, $\phi_i\in H^0(Z,I/I^2)$, is a collection of global sections that generate $I/I^2$ as an $\Os_Z$-module with $I$ being the sheaf of ideals defining $Z$,
\item $g\colon Z\to Y$ is a regular map.
\end{itemize}
We refer to $Z^{red}$ as the support of $c$. For $c=(Z,\phi,g)\in\Frn_n(X,Y\wedge T^l)$ put
\[
\Sigma(c)=(Z\times 0,\Sigma(\phi),g)\in\Frn_{n+1}(X,Y\wedge T^l),
\]
where $Z\times 0$ is the closed subset in $\A^{n+1}_X=\A^n_X\times\A^1$ and $\Sigma(\phi)=(\phi_1,\ldots,\phi_{n+l},t)$ with $t$ being the coordinate function on the last copy of $\A^1$.
\end{dfn}

\begin{dfn}\label{rem:frtonorm}
	For $X,Y\in \Sm_k$ take $(Z,U,(\phi_1,\ldots,\phi_{n+l}),g)\in \Fr_n(X_+,Y_+\wedge T^l)$. Here
\begin{enumerate}
\item 
$Z$ is a closed subset of $\A^n_X$ finite over $X$,
\item
$U$ is an \'etale neighborhood of $Z$ in $\A^n_X$,
\item
$\phi=(\phi_1,\ldots,\phi_{n+l})\colon U\to\A^{n+l}_k$ such that $\phi^{-1}(0)=Z$,
\item
$g\colon U\to Y$ is a regular morphism.
\end{enumerate}
Let $\widetilde{Z}=U\times_{\A^{n+l}_k}\Spec k$ be the scheme-theoretic preimage of $\{0\}$ under the morphism $\phi\colon U\to \A^{n+l}_k$. Note that $Z=\widetilde{Z}^{red}$ whence the composition $\widetilde{Z}\to U\to\A^n_X$ is a closed embedding. Since the projection $U\to \A^n_X$ is \'etale then the canonical morphism of $\Os_{\widetilde{Z}}$-modules $\rho\colon J/J^2\to I/I^2$ is an isomorphism, where $J$ and $I$ are the sheaves of ideals defining $\widetilde{Z}$ in $U$ and $\A^n_X$ respectively. 
The classes $\overline{\phi}_1,\ldots,\overline{\phi}_{n+l}\in H^0(\widetilde{Z},J/J^2)$ generate $H^0(\widetilde{Z},J/J^2)$ as an $\Os_{\widetilde{Z}}$-module thus $\rho(\overline{\phi}_1),\ldots,\rho(\overline{\phi}_{n+l})$ generate $H^0(\widetilde{Z},I/I^2)$ as an $\Os_{\widetilde{Z}}$-module. Put 
\[
\Theta\colon \Fr_n(Y_+,X_+\wedge T^l)\to\Frn_n(Y,X\wedge T^l)
\]
to be the map given by
\[
(Z,U,(\phi_1,\ldots,\phi_{n+l}),g) \mapsto (\widetilde{Z}, \rho(\overline{\phi}_1),\ldots,\rho(\overline{\phi}_{n+l}), g|_{\widetilde{Z}}).
\]
Note that $\Sigma\circ \Theta = \Theta\circ \Sigma$.
\end{dfn}

\begin{dfn}\label{frp} Let $X,Y\in\Sm_k$ and $p\colon Y\to X$ be a finite morphism. Put $l=\dim X-\dim Y$. 
	
Denote $\Fr_n(X\xleftarrow{p} Y)$ the subset of $\Fr_n(X_+,Y_+\wedge T^l)$ consisting of $(Z,U,\phi,g)$ satisfying:
\begin{enumerate}
	\item 
	The diagram
	\[
	\xymatrix{
	U \ar[d]_\pi \ar[r]^{g} & Y \ar[d]^p \\
	\A^n_X \ar[r]^{\pi_X} &	X,
}
	\]
	commutes. Here $\pi$ and $\pi_X$ are the canonical projections.
	\item
	For the scheme-theoretic preimage $\widetilde{Z}=\phi^{-1}(0)$ the morphism $	g\colon \widetilde{Z}\to Y$	is an isomorphism.
\end{enumerate}
We refer to the elements of $\Fr_n(X\xleftarrow{p} Y)$ as framed enhancements of $p$.

Denote $\Frn_n(X\xleftarrow{p} Y)$ the subset of $\Frn_n(X,Y\wedge T^l)$ consisting of $(Z,\phi,g)$ satisfying:
\begin{enumerate}
	\item 
	The diagram
	\[
	\xymatrix{
		Z \ar[d]_i \ar[r]^{g} & Y \ar[d]^p \\
		\A^n_X \ar[r]^{\pi_X} &	X,
	}
	\]
	commutes. Here $i$ is the closed embedding and $\pi_X$ is the projection.
	\item
	The morphism $g\colon Z\to Y$ is an isomorphism.
\end{enumerate}
In other words, the elements of $\Frn_n(X\xleftarrow{p} Y)$ are pairs $(i,\phi)$ with $i\colon Y\to\A^n_X$ being a closed embedding over $X$ and $\phi$ being a trivialization of the (co-)normal bundle of $i$.

Note that $\Theta$ from Definition~\ref{rem:frtonorm} can be restricted to a map 
\[
\Theta_p\colon \Fr_n(X\xleftarrow{p}Y)\to \Frn_n(X\xleftarrow{p}Y).
\]
\end{dfn}

\begin{dfn}\label{dfn:det}
	Let $p\colon Y\to X$ be a finite map with $X,Y\in\Sm_k$. For $a\in\Fr_n(X\xleftarrow{p} Y)$ we denote
	\[
	\det a \colon \omega_{Y/X}\cong \det N_i\xrightarrow{\simeq} \Os_Y
	\]
	the trivialization induced by the trivialization $\phi$ with $\Theta_p(a)=(i,\phi)$.
\end{dfn}

\begin{lem}\label{lem:enhancefunct}
Let $X,Y,V\in \Sm_k$ and $p\colon V\to Y$, $q\colon Y\to X$ be finite morphisms. Put $l=\dim X-\dim Y$. Then for $a\in \Fr_n(X\xleftarrow{p} Y), \, b\in \Fr_m(Y\xleftarrow{q} V)$ one has
\[
(b\boxtimes \id_{T^{l}})\circ a\in\Fr_{n+m}(X\xleftarrow{q\circ p} V).
\]
\end{lem}
\begin{proof}
Straightforward.
\end{proof}

\begin{lem}\label{lem:frameeqsur}
Let $X,Y\in\Sm_k$ and $p\colon Y\to X$ be a finite morphism. Suppose that $Y$ is affine and that there is an \'etale map $\pi\colon Y\to\A^d_k$. Then the map
\[
\Theta_p\colon \Fr_n(X\xleftarrow{p} Y)\to\Frn_n(X\xleftarrow{p} Y)
\]
introduced in Definition~\ref{rem:frtonorm} is surjective.
\end{lem}
\begin{proof}
Choose $c=(Z,\phi,g)\in\Frn_n(X\xleftarrow{p} Y)$ and let $I\subset k[\A^n_X]$ be the ideal corresponding to $Z$. Take an arbitrary lift of 
$
\phi=(\phi_1,\ldots,\phi_{n+l})\in H^0(\A^n_X,I/I^2)
$
to
$
\psi=(\psi_1,\ldots,\psi_{n+l})\in H^0(\A^n_X,I).
$
Note that
$
I=\langle \psi_1,\ldots,\psi_{n+l}\rangle + I^2
$
yields that $(\psi_1,\ldots,\psi_{n+l})$ defines $Z$ in some affine Zariski neighborhood $U'$ of $Z$ in $\A^n_X$.

The composition $\pi\circ g\colon Z\to \A^d_k$ can be extended to a morphism $g'\colon U'\to \A^d_k$ making the following diagram commute.
\[
\xymatrix{
Z \ar[r]^g \ar[d] & Y \ar[d]^{\pi} \\
U' \ar[r]_{g'}&  \A^d_k
}
\]
Put $U''=U'\times_{\A^d_k} Y$. The morphism $U''\to U'$ is \'etale since $Y\to \A^d_k$ is \'etale. Moreover, $Z\times_{U'}U''=Z\sqcup Z'$ whence $U=U''-Z'$ is an \'etale neighborhood of $Z$ in $\A^n_X$. Then for the projections $\rho_{U'}\colon U\to U'$ and $\rho_{Y}\colon U\to Y$ we have
\[
\Theta_p(Z,U,\psi\circ \rho_{U'},\rho_Y)= (Z,\phi,g). \qedhere
\]
\end{proof}

\begin{lem}\label{lem:frameeqhom} Under the assumptions of Lemma~\ref{lem:frameeqsur} suppose that $\Theta_p(a)=\Theta_p(b)$
for some $a,b\in \Fr_n(X\xleftarrow{p} Y)$. Then there exists 
\[
H\in\Fr_n(X\times\A^1\xleftarrow{p\times\id_{\A^1}} Y\times\A^1)
\]
such that $H|_{X\times \A^1}=a$ and $H|_{X\times \A^1}=b$.
\end{lem}
\begin{proof}
Let $a=(Z,U,\phi,g)$ and $b=(Z,U',\phi',g')$. Passing to the \'etale neighborhood $U\times_{\A^n_X}U'$ of $Z$ we may assume that $U=U'$. Then both the collections of regular functions $\phi,\phi'$ generate the same ideal
\[
I=\langle \phi_1,\ldots,\phi_{n+l} \rangle=\langle \phi'_1,\ldots,\phi'_{n+l}\rangle\text{ in }k[U]
\]
and, moreover, 
\[
\phi_i=\phi'_i \mod I^2,\,i=1,\hdots, n+l.
\]
Put $\Phi_i=\phi_i+t(\phi'_i-\phi_i)\in k[U\times\A^1]$. Let $J$ denote the ideal $J=I\otimes_kk[t]$ in $k[U\times\A^1]$. Then the elements $\Phi_i$ lie in $J$ and $\Phi_i=\phi_i \mod J^2$. Let $J'$ denote the ideal $J'=\langle\Phi_1,\ldots,\Phi_{n+l}\rangle$. Then $J\cdot J/J'=J/J'$ and by Nakayama's lemma $r\cdot J/J'=0$ for some element $r$ in $k[U\times\A^1]$ such that $r=1\mod J$. Then $U\times\A^1-Z(r)$ is an open neighborhood of $Z\times\A^1$ and we have
\[
Z(\Phi_1,\hdots, \Phi_{n+l})\cap(U\times\A^1-Z(r))=(Z\times \A^1),
\] so
\[
Z(\Phi)=Z(\Phi_1,\hdots, \Phi_{n+l})=(Z\times \A^1) \sqcup \mathcal{Z}'.
\]
Put $\mathcal{U}'=(U\times \A^1) - \mathcal{Z}'$. Note that $\mathcal{Z}'\cap (U\times \{0,1\})=\emptyset$, i.e. $U\times \{0,1\}\subset \mathcal{U}'$.

Take an \'etale map $\pi\colon Y\to\A^d_k$. Similar to the proof of the previous Lemma we can extend 
\[
(\pi\times \id_{\A^1})\circ (g\sqcup g') \colon U\times \{0,1\}\to \A^d_k\times \A^1
\]
to a map $G\colon \mathcal{U}'\to \A^d_k\times \A^1$ making the following diagram commute.
\[
\xymatrix{
U\times\{0,1\} \ar@/^1.5pc/[drr]_{g\sqcup g'}  \ar@/_1.5pc/[ddr]^i \ar[dr] & & \\
& \mathcal{U}'' \ar[r]_(0.4){\rho_{Y\times \A^1}} \ar[d]  & Y\times \A^1 \ar[d]^{\pi\times \id_{\A^1}}\\
& \mathcal{U}' \ar[r]^(0.4){G} & \A^d_k\times \A^1
}
\]
Here $i$ is the closed embedding and $\mathcal{U}''=\mathcal{U}'\times_{\A^d_k\times\A^1}(Y\times\A^1)$.  We have 
\[
(Z\times \A^1)\times_{\mathcal{U}'} {\mathcal{U}''} = (Z\times \A^1)\sqcup \mathcal{Z}''.
\]
Put $\mathcal{U}=\mathcal{U}''-\mathcal{Z}''$. Then 
\[
H=(Z\times\A^1,\mathcal{U},\Phi|_{\mathcal{U}},\rho_{Y\times \A^1}|_{\mathcal{U}})\in \Fr_n(X\times\A^1\xleftarrow{p\times \id_{\A^1}} Y\times\A^1)
\]
gives the claim.
\end{proof}

\begin{dfn} 
	For $X,Y\in\Sm_k$ and a finite map $p\colon Y\to X$ we say that
	$a,b \in \Fr_n(X\xleftarrow{p} Y)$ (resp. $\Frn_n(X\xleftarrow{p}Y)$) are $\A^1$-homotopic if there exists $H\in \Fr_n(X\times\A^1\xleftarrow{p\times \id_{\A^1}} Y\times\A^1)$ (resp. $\Frn_n(X\times\A^1\xleftarrow{p\times \id_{\A^1}} Y\times\A^1)$) such that 
	\[
	H|_{X\times \{0\}}=a, \quad H|_{X\times \{1\}}=b.
	\]
	Denote by $\overline{\Fr}_n(X\xleftarrow{p} Y)$ (resp. $\hFrn_n(X\xleftarrow{p} Y)$) the set of equivalence classes of $\A^1$-homotopic correspondences.
\end{dfn}

\begin{cor}\label{lem:frameeq} Let $X,Y\in\Sm_k$ and $p\colon Y\to X$ be a finite morphism. Suppose that $Y$ is affine and that there is an \'etale map $\pi\colon Y\to\A^d_k$, then the map introduced in Definition~\ref{rem:frtonorm} induces a bijection
\[
\Theta_p\colon \overline{\Fr}_n(X\xleftarrow{p} Y)\xrightarrow{\simeq} \hFrn_n(X\xleftarrow{p} Y).
\]
\end{cor}
\begin{proof}
	Follows from Lemmas~\ref{lem:frameeqsur} and~\ref{lem:frameeqhom}.
\end{proof}

\begin{dfn} 
Let $X,Y\in\Sm_k$ and $p\colon Y\to X$ be a finite morphism. Put $l=\dim X-\dim Y$. For $A\in\mathrm{GL}_{n+l}(k[Y])$ and $(i,\phi)\in \Frn_n(X\xleftarrow{p} Y)$ put
\[
(i,\phi)\cdot A = (i,\phi\cdot A) \in \Frn_n(X\xleftarrow{p} Y).
\]
Here on the right-hand side we take the matrix multiplication of the row $\phi=(\phi_1,\hdots, \phi_{n+l})$ by the matrix $A$.
\end{dfn}


\begin{lem}\label{lem:movemb}
Let $X,Y\in\Sm_k$ and $p\colon Y\to X$ be a finite morphism. Put $l=\dim X-\dim Y$. Suppose that $Y$ is affine. Then for every $(i,\phi),\,(j,\psi)\in \Frn_n(X\xleftarrow{p}Y)$ there exists a matrix $A\in \operatorname{GL}_{2n+l}(k[Y])$ such that 
\[
\overline{\Sigma^n(i,\phi)} = \overline{\Sigma^n(j,\psi)\cdot A} \in \hFrn_{2n}(X\xleftarrow{p}Y).
\]
\end{lem}
\begin{proof}
Consider the map 
\[
I\colon Y\times\A^1\to \A^n_X\times_X \A^n_X\times \A^1 = \A^{2n}_X\times\A^1
\]
given by the formula 
\[
I(y,t)=((1-t)i(y)+t j(y),t(t-1)i(y),t).
\]
This morphism is a morphism over $X\times \A^1$ whence $I$ is finite, moreover, it is clear that $I$ is a monomorphism. Thus $I$ is a closed embedding.
Denote by $N_I$ the normal bundle for the embedding. Lindel theorem~\cite{Lindel} yields that 
\[
N_I \cong \pi^*(N_I|_{Y\times \{0\}})
\]
with $\pi\colon Y\times\A^1\to Y$ being the projection. Since $I|_{Y\times \{0\}}$ is the suspension of inclusion $i$, the pullback of the suspension of $\phi$ gives a trivialization of $N_I$ which we denote $\Phi$. We obtained
\[
(I,\Phi) \in \Frn_{2n}(X\times \A^1 \xleftarrow{p\times\id_{\A^1}} Y\times \A^1)
\]
with $(I,\Phi)|_{X\times\{0\}}=\Sigma^n(i,\phi)$ and $(I,\Phi)|_{X\times\{1\}}=(j\times \{0\},\psi')$ for some trivialization $\psi'$ of the normal bundle for the embedding $j\times\{0\} \colon Y\to \A^n_X\times \A^n$. The claim follows, since every two trivializations differ by some matrix $A\in\mathrm{GL}_{2n+l}(k[Y])$.
\end{proof}

\section{Pushforward maps: finite morphisms}\label{pfmocohm}

\begin{dfn}		
For $X\in\Sm_k$, a closed subset $S\subset X$ and a sheaf of abelian groups $M$ the Cousin complex $C^\bullet_S(X,M)$ is the complex 
\[
\bigoplus_{x\in X^{(0)}\cap S}H^0_x(X,M) \to \bigoplus_{x\in X^{(1)}\cap S}H^1_x(X,M) \to \hdots \to \bigoplus_{x\in X^{(d)}\cap S}H^d_x(X,M)
\]
with the terms given by
\[
C^n_S(X,M)=\bigoplus_{x\in X^{(n)}\cap S}H^n_x(X,M)
\]
and the differential induced by the colimit of the connecting homomorphisms 
\[
H^n_{Z-Z'}(X-Z',M)\to H^{n+1}_{Z'-Z''}(X-Z'',M)
\]
for closed subsets $Z''\subset Z'\subset Z$~\cite[\S 1.2]{CTHK}. For a homotopy module $M_*$ \cite[Theorem 5.1.10]{CTHK} yields that there are canonical isomorphisms
\[
H^n_S(X,M_*)\cong H^n(C^\bullet_S(X,M)).
\]
\end{dfn}

\begin{dfn}\label{dfn:notation}
	For this section fix the following notation: $X,Y\in \Sm_k$, $p\colon Y\to X$ is a finite morphism, $l=\dim X-\dim Y$, $M_*$ is a homotopy module.
\end{dfn}

\begin{dfn}\label{dfn:pushcohfin}
In the notation of Definition~\ref{dfn:notation} let $Z\subset X$ be a closed subset. Then for $a\in\Fr_n(X\xleftarrow{p} Y)$ in the notation of Definition~\ref{frcohaction} we have $a^{-1}(p^{-1}(Z))=Z$ and get a homomorphism
\[
a^*\colon H^{n-l}_{p^{-1}(Z)}(Y,M_{*-l})\to H^n_{Z}(X,M_*).
\]
\end{dfn}


\begin{lem}\label{lem:complexmorphism} In the notation of Definition~\ref{dfn:notation} let $S\subset X$ be a closed subset. Then for every $a\in\Fr_n(X\xleftarrow{p} Y)$ the homomorphisms from Definition~\ref{dfn:pushcohfin} give rise to a morphism of complexes
\[
a^*\colon C^\bullet_{p^{-1}(S)}(Y,M_{*-l})[-l]\to C^\bullet_{S}(X,M_*).
\]
This morphism depends only on the class of $a$ in $ \overline{\Fr}_n(X\xleftarrow{p} Y)$.
\end{lem}
\begin{proof}
For closed embeddings $Z''\subset Z'\subset Z \subset X$ the differentials
\begin{gather*}
H^n_{p^{-1}(Z-Z')}(Y-p^{-1}(Z'),M_*)\to H^{n+1}_{p^{-1}(Z'-Z'')}(Y-p^{-1}(Z''),M_*),\\
H^n_{Z-Z'}(X-Z',M_*)\to H^{n+1}_{Z'-Z''}(X-Z'',M_*)
\end{gather*}
are induced by the maps 
\begin{gather*}
\Z[(Y-p^{-1}(Z''))/(Y-p^{-1}(Z'))]\to \Z[(Y-p^{-1}(Z'))/(Y-p^{-1}(Z))][1], \\
\Z[(X-Z'')/(X-Z')]\to \Z[(X-Z')/(X-Z)][1]
\end{gather*}
respectively. Since $a^*$ is induced by a morphism of sheaves (see Definition~\ref{dfn:pushcohfin}) then it fits into the following commutative diagram.
\[
\xymatrix{
H^{n-l}_{p^{-1}(Z-Z')}(Y-p^{-1}(Z'),M_{*-l})\ar[d]^{a^*}\ar[r] & H^{n+1-l}_{p^{-1}(Z'-Z'')}(Y-p^{-1}(Z''),M_{*-l})\ar[d]^{a^*}\\
H^n_{Z-Z'}(X-Z',M_*)\ar[r] & H^{n+1}_{Z'-Z''}(X-Z'',M_*).
}
\]
The colimit with respect to $Z''\subset Z'\subset Z \subset X$ gives the desired morphism of complexes. Since $M_*$ is strictly homotopy invariant, $a^*$ depends only on the class in $ \overline{\Fr}_n(X\xleftarrow{p} Y)$.
\end{proof}

\begin{rem} In the notations of the above lemma the assignment 
\[
U\mapsto C^n_{p^{-1}(S)}(p^{-1}(U),M_*)
\]
defines a sheaf on the small Zariski site of $X$ and $a^*$ gives rise to a morphism of complexes
\[
a^*\colon C^\bullet(p^{-1}(-),M_{*-l})[-l]\to C^\bullet_S(-,M_*).
\]
\end{rem}

\begin{dfn} 
Let $M_*$ be a homotopy module. Then Definition~\ref{KMWmodule} endows $M_*$ with a $\underline{\mathrm{GW}}$-module structure. For $Y\in\Sm_k$ and a line bundle (invertible sheaf) $\mathcal{L}$ on $Y$ define the twist $M_*\otimes \mathcal{L}$ as the sheaf on $Y$ given by the tensor product
\[
M_*\otimes L=M_*\otimes_{\Z[\Os^{\times}_Y]}\Z[\mathcal{L}^{\times}].
\]
Here
\begin{itemize}
	\item 
	$\Os^{\times}_Y$ is the sheaf of invertible regular functions,
	\item
	$\mathcal{L}^{\times}$ is the sheaf of nowhere vanishing sections of $\mathcal{L}$,
	\item
	the module structure $M_*\times \Z[\Os^{\times}_Y] \to M_*$ is induced by the $\underline{\mathrm{GW}}$-module structure combined with the morphism $\Os^{\times}_Y \to \underline{\mathrm{GW}}$, $u\mapsto \langle u\rangle$.
\end{itemize}
We are mostly interested in the particular case of
\[
	\mathcal{L} =\omega_{Y/X}=\omega_{Y/k}\otimes p^*\omega_{X/k}^{-1}
\]
for $X,Y\in\Sm_k$ and a morphism $p\colon Y\to X$. 
\end{dfn}

%
%

\begin{dfn}
In the notation of Definition~\ref{dfn:notation} for a closed subset  $S\subset X$, a line bundle $\mathcal{L}$ on $X$, a trivialization $\theta\colon \mathcal{L}\xrightarrow{\simeq} \Os_X$ and $a\in \Fr_n(X\xleftarrow{p} Y)$ put
\[
p^{a,\theta}_*= (\otimes \theta^{-1}) \circ a^* \circ(\otimes \det a) \circ (\otimes p^*\theta)
\]
for the composition
\begin{multline*}
C^\bullet_{p^{-1}(S)}(Y,M_{*-l}\otimes\omega_{Y/X}\otimes p^*\mathcal{L})[-l]\xrightarrow{\otimes p^*\theta}
 C^\bullet_{p^{-1}(S)}(Y,M_{*-l}\otimes\omega_{Y/X})[-l]\xrightarrow{\otimes \det a} \\
 \xrightarrow{\otimes \det a} C^\bullet_{p^{-1}(S)}(Y,M_{*-l})[-l]\xrightarrow{a^*} C^\bullet_S(Y,M_*)\xrightarrow{\otimes \theta^{-1}}
 C^\bullet_S(Y,M_*\otimes \mathcal{L})
\end{multline*}
Here $\det a\colon \omega_{Y/X}\xrightarrow{\simeq} \Os_{Y}$ is given by Definition~\ref{dfn:det}.
\end{dfn}

\begin{lem}\label{lem:finpushchoice}
	 In the notation of Definition~\ref{dfn:notation} let $\mathcal{L}$ be a trivial line bundle on $X$. Suppose that $Y$ is affine an that there exists an \'etale map $Y\to \A^d_k$. Then for $a,b\in\Fr_n(X\xleftarrow{p} Y)$ and $\theta,\theta'\colon \mathcal{L} \xrightarrow{\simeq} \Os_X$ we have
\[
p^{a,\theta}_*=p_*^{b,\theta'}\colon C^\bullet_{p^{-1}(S)}(Y,M_{*-l}\otimes\omega_{Y/X})[-l]\to C^*_S(X,M_*).
\]
\end{lem}
\begin{proof}
	The equality $p^{a,\theta}_*=p^{a,\theta'}_*$ follows from Lemma~\ref{frcomextprod} claiming $a^*$ to be a homomorphism of $\underline{\mathrm{GW}}$-modules. The equality $p^{a,\theta'}_*=p^{b,\theta'}_*$ follows from the fact that $a^*$ depends only on the $\A^1$-homotopy class of $a$ and Corollary~\ref{lem:frameeq} combined with Lemmas~\ref{matrixdet} and~\ref{lem:movemb}.
\end{proof}


\begin{lem}\label{semilocetale} For $X\in \Sm_k$ and a finite collection of points $x_1,\ldots, x_n\in X$ there exists a Zariski open $U\subset X, x_i\in U, i=1,\hdots n,$ and an \'etale map $f\colon U\to\A^d_k$ with $d=\dim X$.
\end{lem}
\begin{proof}
	A general projection $U\to \A^d_k$ of an affine Zariski open $U\subset X, x_i\in U, i=1,\hdots n,$ is \'etale at $x_1,\ldots, x_n$ whence the claim.
\end{proof}

\begin{dfn}\label{dfn:finpush}
	In the notation of Definition~\ref{dfn:notation} let $S\subset X$ be a closed subset and $\mathcal{L}$ be a line bundle on $X$. Define the pushforward map 
\[
p_*\colon C^\bullet_{p^{-1}(S)}(Y,M_{*-l}\otimes\omega_{Y/X}\otimes p^*\mathcal{L})[-l]\to C^\bullet_S(X,M_*\otimes \mathcal{L})
\]
as follows. By Lemma~\ref{semilocetale} we can find an open cover $\bigcup_{i=1}^n U_i=X$ such that for every $i$ the following holds:
\begin{enumerate}
\item
$U_i$ is affine,
\item
$\mathcal{L}|_{U_i}$ is trivial,
\item
$V_i=p^{-1}(U_i)$ admits an \'etale map $V_i\to\A^d_k$,
\item
there exists a closed embedding $V_i\to\A^n_{U_i}$ over $U_i$ with a trivial normal bundle.
\end{enumerate}
These conditions imply that we can choose some $a_i\in\Fr_n(U_i\xleftarrow{p_i} V_i)$ together with trivializations $\theta_i\colon \mathcal{L}|_{U_i}\xrightarrow{\simeq} \Os_{U_i}$ with $p_i=p|_{V_i}$. Lemma~\ref{lem:finpushchoice} yields that the maps
\[
p_{i*}=(p_i^{a_i,\theta_i})_*\colon C^\bullet(V_i,M_{*-l}\otimes\omega_{V_i/U_i}\otimes p^*(\mathcal{L}))[-l]\to 
C^\bullet(U_i,M_*\otimes \mathcal{L})
\]
do not depend on the choice of $a_i$ and $\theta_i$ and, in particular, coincide on the intersections $V_{ij}=V_i\cap V_j$. Then $p_{i*}$-s give rise to the pushforward map $p_*$ making the following diagram commute.
\[
\xymatrix{
C^\bullet_{p^{-1}(S)}(Y,M_{*-l}\otimes\omega_{Y/X}\otimes p^*\mathcal{L})[-l]\ar@{^{(}->}[r]\ar[d]^{p_*} & \bigoplus\limits_{i=1}^n C^\bullet_{p^{-1}(S)}(V_i,M_{*-l}\otimes\omega_{V_i/U_i}\otimes p_i^*\mathcal{L})[-l] \ar[d]^{\oplus p_{i*}}\\
C^\bullet_S(X,M_*\otimes \mathcal{L})\ar@{^{(}->}[r] & \bigoplus\limits_{i=1}^n C^*_S(U_i,M_*\otimes \mathcal{L})
}
\]
It also follows from Lemma~\ref{lem:finpushchoice} that $p_*$ does not depend on the cover.
\end{dfn}

\begin{lem}\label{pffinite} 
	In the notation of Definition~\ref{dfn:notation} let $V\in \Sm_k$ and $q\colon V\to Y$ be a finite morphism. Put $l'=\dim X-\dim Y$. Then for a line bundle $\mathcal{L}$ over $X$ and a closed subset $S\subset X$  one has
\[
(p\circ q)_*=p_*\circ q_*\colon 
C^\bullet_{(p\circ q)^{-1}(S)}(V,M_{*-l-l'}\otimes\omega_{V/X}\otimes (p\circ q)^*\mathcal{L})[-l-l']\to C^\bullet_S(X,M_*\otimes \mathcal{L}).
\]
\end{lem}
\begin{proof}
Follows from Lemma~\ref{lem:enhancefunct} and the construction.
\end{proof}

\begin{lem}\label{lem:trsupiso}
	In the notation of Definition~\ref{dfn:notation} for a closed subset $S\subset Y$ the pushforward map 
\[
i_*\colon C^\bullet_S(Y,M_{*-l}\otimes\omega_{Y/X})[-l]\to C^\bullet_S(X,M_*)
\]
is an isomorphism.
\end{lem}
\begin{proof}
Since $i_*$ is a morphism of complexes of Zariski sheaves we may shrink $X$. Thus we may assume that there exists $a\in\Fr_0(X\xleftarrow{i} Y)$ that induces an isomorphism of sheaves $X/(X-Y)\xrightarrow{\simeq} Y_+\w T^l$. Then for every closed $Z\subset Y$ the map $a^*\colon H^{i-l}_Z(Y,M_{*-l})\to H^{i}_Z(X,M_*)$ is an isomorphism. Taking the colimit along $Z$ we obtain the claim.
\end{proof}

\begin{cor}\label{cor:ix}
For $X\in \Sm_k$, $x\in X^{(n)}$, a line bundle $\mathcal{L}$ on $X$ and a homotopy module $M_*$ the pushforward map
\[
(i_x)_*\colon H^0(x,M_{*-n}\otimes\omega_{x/X}\otimes i_x^*\mathcal{L}) \xrightarrow{\simeq}
H^{n}_{x}(X,M_{*}\otimes \mathcal{L})
\]
for the embedding $i_x\colon x\to X$ is an isomorphism.
\end{cor}
\begin{proof}
	There is a Zariski open $x\in U\subset X$ such that $S=\overline{x}\cap U$ is smooth. Lemma~\ref{lem:trsupiso} applied to the closed embedding $i\colon S\to U$ yields an isomorphism
	\[
	i_*\colon C^\bullet_S(S,M_{*-n}\otimes\omega_{S/U}\otimes i^*\mathcal{L})[-n]\to C^\bullet_S(U,M_*)
	\]
	which restricts to the desired isomorphism.
\end{proof}

\begin{cor}\label{killtwist} For $X\in\Sm_k$, a closed subset $S\subset X$, a line bundle $\mathcal{L}$ on $X$ and a homotopy module $M_*$ the pushforward map
\[
(i_\mathcal{L})_*\colon H^{n}_S(X,M_{*}\otimes \mathcal{L})\to H^{n+1}_S(L,M_{*+1})
\]
is an isomorphism. Here $L$ is the total space of $\mathcal{L}$ and $i_{\mathcal{L}}\colon X\to L$ is the zero section. Moreover, if $\theta\colon  \mathcal{L}\xrightarrow{\simeq} \Os_X$ is a trivialization of $\mathcal{L}$ then the following diagram commutes.
\[
\xymatrix{
H^{n}_S(X,M_{*}\otimes \mathcal{L}) \ar[r]^{(i_{\mathcal{L}})_*} \ar[d]^{\otimes \theta} & H^{n+1}_S(L,M_{*+1}) \ar[d]^{\otimes \theta} \\
H^{n}_S(X,M_{*}) \ar[r]^{(i_{\A^1})_*} & H^{n+1}_S(\A^1_X,M_{*+1})
}
\]
\end{cor}
\begin{proof}
Follows from Lemma~\ref{lem:trsupiso}. The commutativity of the diagram immediately follows from the definition of the pushforward map.
\end{proof}

\section{Pushforward maps: general case}\label{sect:pushgen}
\begin{dfn}\label{pfcompdef}
	Let $f\colon Y\to X$ be a morphism in $\Sm_k$ of relative dimension $d=\dim Y-\dim X$ and let $\mathcal{L}$ be a line bundle on $X$. For $y\in Y^{(n)}$ such that $f\colon y\to f(y)$ is finite define $(f_*)_y$ to be the morphism making the following diagram commute.
	\[
	\xymatrix{
	H^n_y(Y,M_*\otimes\omega_{Y/X}\otimes f^*\mathcal{L}) \ar[r]^{(f_*)_y} &  H^{n-d}_{f(y)}(X,M_{*-d}\otimes \mathcal{L}) \\
	H^0(y,M_{*-n}\otimes\omega_{y/X}\otimes (f\circ i_y)^*\mathcal{L}) \ar[u]^{(i_y)_*}_{\cong} \ar[r]^{(f|_y)_*} &
	H^0(f(y),M_{*-n}\otimes\omega_{f(y)/X}\otimes i_{f(y)}^*\mathcal{L}) \ar[u]_{(i_{f(y)})_*}^{\cong}
}
	\]
	Here the vertical isomorphisms are given by Corollary~\ref{cor:ix} and the bottom morphism is the pushforward morphism from Definition~\ref{dfn:finpush} for the finite morphism $f|_y\colon y\to f(y)$.
\end{dfn}
%

\begin{prop} \label{prop:push}
	In the notation of Definition~\ref{pfcompdef} let $S$ be a closed subset of $Y$ finite over $X$. Then the maps $(f_*)_y$ for $y\in S$ define a morphism of complexes
\[
f_*\colon C^\bullet_S(Y,M_*\otimes\omega_{Y/X}\otimes f^*\mathcal{L})\to C^\bullet_{f(S)}(X,M_{*-d}\otimes \mathcal{L})[-d].
\]
\end{prop}

\begin{proof} Consider the case of trivial $\mathcal{L}$, the proof in the general case is the same. 
	
Take $x\in X^{(n-d)}\cap f(S), x'\in X^{(n-d+1)}\cap \overline{x}$,
choose a point $y\in f^{-1}(x)\cap S$ and let $y'_1,y'_2,\hdots, y'_m$ be the preimages of $x'$ in $\overline{y}$. We need to check that the diagram
\[
\xymatrix{
H^n_y(Y,M_*\otimes\omega_{Y/X})\ar[r]\ar[d]^{(f_*)_y}& \bigoplus\limits_{i=1}^m H^{n+1}_{y'_i}(Y,M_*\otimes\omega_{Y/X})\ar[d]^{\sum(f_*)_{y'_i}}\\
H^{n-d}_x(X,M_{*-d})\ar[r] & H^{n-d+1}_{x'}(X,M_{*-d})
}
\]
commutes. Let $Z\to \overline{y}$ be the normalization. There exists Zariski open subsets $U\subset Y$ and $V\subset X$ such that
\begin{multicols}{2}
\begin{enumerate}
	\item 
	$y_i'\in U$ for all $i$,
	\item
	$x'\in V$,
	\item
	$\widetilde{Z}=Z\times_Y U$ is smooth,
	\item
	$f(\overline{y}\cap U)=\overline{x}\cap V$.
\end{enumerate}
\end{multicols}
\noindent
Then $p\colon\widetilde{Z}\to U$ and $f\circ p\colon\widetilde{Z}\to V$ are finite maps between smooth varieties and Definition~\ref{dfn:finpush} gives morphisms of complexes 
\begin{gather*}
p_*\colon C^\bullet(\widetilde{Z},M_{*-n}\otimes\omega_{\widetilde{Z}/X})[-n]\to C^\bullet(U,M_*\otimes\omega_{U/X}),\\ 
(f\circ p)_*\colon C^\bullet(\widetilde{Z},M_{*-n}\otimes\omega_{\widetilde{Z}/X})[-n]\to C^\bullet(V,M_{*-d})[-d].
\end{gather*}
Then the squares in the diagram
\[
\xymatrix{
H^n_y(U,M_*\otimes\omega_{U/X})\ar[r]\ar@/_5pc/[dd]_{(f_*)_y}& \bigoplus\limits_{i=1}^m H^{n+1}_{y'_i}(U,M_*\otimes\omega_{U/X})\ar@/^5pc/[dd]^{\sum_i (f_*)_{y'_i}}\\
H^0_y(\widetilde{Z},M_{*-n}\otimes\omega_{\widetilde{Z}/X})\ar[d]^{((f\circ p)_*)_y}\ar[u]_{(p_*)_y}\ar[r] & \bigoplus\limits_{j=1}^l H^1_{y''_j}(\widetilde{Z},M_{*-n}\otimes\omega_{\widetilde{Z}/X})\ar[d]_{\sum_j((f\circ p)_*)_{y''_j}}\ar[u]^{\sum_j(p_*)_{y''_j}}\\
H^{n-d}_x(V,M_{*-d})\ar[r] & H^{n-d+1}_{x'}(V,M_{*-d})
}
\]
commute. Here $y_j''$ denote the points of $\widetilde{Z}$ lying over $y'_i$. 
The triangles commute by Lemma~\ref{pffinite}. The claim follows since
\[
(p_*)_y\colon H^0_y(\widetilde{Z},M_{*-n}\otimes\omega_{\widetilde{Z}/X})\to H^n_y(U,M_*\otimes\omega_{U/X})
\]
is an isomorphism.
\end{proof}

\begin{dfn}\label{dfn:pushc}
In the notation of Definition~\ref{pfcompdef} let $S$ be a closed subset of $Y$ finite over $X$. Then applying Proposition~\ref{prop:push} and taking cohomology we obtain pushforward maps
\[
f_*\colon H^n_S(Y,M_*\otimes\omega_{Y/X}\otimes f^*\mathcal{L}) \to 
H^{n-d}_{f(S)}(X,M_{*-d}\otimes \mathcal{L}).
\]
The constructed pushforward maps are functorial by Lemma~\ref{pffinite}.
\end{dfn}

\begin{lem}\label{lem:opencommute} 
In the notation of Definition~\ref{pfcompdef} let $S$ be a closed subset of $Y$ finite over $X$. Then for an open embedding $j_X\colon X_0\xrightarrow{} X$ the square
\[
\xymatrix{
H^n_S(Y,M_*\otimes\omega_{Y/X}\otimes f^*\mathcal{L})\ar[r]^(0.43){j_Y^*}\ar[d]^{f_*} & H^n_{S_0}(Y_0,M_*\otimes\omega_{Y_0/X_0}\otimes (f\circ j_Y)^*\mathcal{L})\ar[d]^{f_*}\\
H^{n-d}_{f(S)}(X,M_{*-d}\otimes \mathcal{L})\ar[r]^(0.43){j_X^*} & H^{n-d}_{f(S_0)}(X_0,M_{*-d}\otimes j_X^*\mathcal{L})
}
\]
commutes, where $j_Y\colon Y_0=X_0\times_X Y \to Y$ is the open embedding and $S_0=X_0\times_X S$.
\end{lem}
\begin{proof}
The claim follows from the commutativity of the following diagram consisting of Cousin complexes which is immediate.
\[
\xymatrix{
	C^\bullet_S(Y,M_*\otimes\omega_{Y/X}\otimes f^*\mathcal{L}) \ar[d]^{f_*} \ar[r]^(0.43){j^*_Y} & C^\bullet_{S_0}(Y_0,M_*\otimes\omega_{Y_0/X_0}\otimes (f\circ j_Y)^*\mathcal{L}) \ar[d]^{f_*} \\
	C^\bullet_{f(S)}(X,M_{*-d}\otimes \mathcal{L})[-d] \ar[r]^(0.43){j^*_X} & C^\bullet_{f(S_0)}(X_0,M_{*-d}\otimes j_X^*\mathcal{L})[-d]
}
\]
Here $j^*_X$ is identity on $H^i_x(X,M_{*-d}\otimes\mathcal{L})$ if $x\in X_0$ and zero otherwise, and similarly for $j^*_Y$.
\end{proof}

\begin{rem}\label{coincideMorel}
It seems that the pushforward maps constructed in Definition~\ref{pfcompdef} coincide with the ones introduced in~\cite[Chapter~5]{Morel}. We do not need this fact in this paper so we do not check the details.
\end{rem}

\begin{dfn}\label{goodemb} 
	Let $X,Y\in \Sm_k$. We say that 
	\begin{enumerate}
		\item 
		a closed embedding $i\colon Y\to X$ admits a framed enhancement if $\Fr_0(X\xleftarrow{i} Y)\neq \emptyset$,
		\item
		a morphism $f\colon Y\to X$ has a framed enhancement if for some $m$ there exists a Zariski open subset $U\subset \A^{m}_X$ and a closed embedding $i\colon Y\to U$ over $X$ such that $i$ admits a framed enhancement.
	\end{enumerate}
\end{dfn}

\begin{rem}
If $f\colon Y\to X$ admits a framed enhancement then
\begin{enumerate}
\item
for every $V\in\Sm_k$ the morphism $f\times\id_V\colon Y\times V\to X\times V$ admits a framed enhancement,
	\item 
for any Zariski open subset $Y_0\subset Y$ the morphism $f|_{Y_0}\colon Y_0\to X$ admits a framed enhancement.
\end{enumerate}
\end{rem}

\begin{lem}\label{semilocalgood} 
For $X\in\Sm_k$ and any finite collection of points $x_1,\hdots, x_n\in X$ there exists a Zariski open $U\subset X$ such $x_i\in U$ for all $i$ and the morphism $U\to\Spec k$ admits a framed enhancement.
\end{lem}
\begin{proof}
Applying Lemma~\ref{semilocetale} choose a Zariski open $U\subset X$ such $x_i\in U$ for all $i$ and that there exists an \'etale map $g\colon U\to\A^d_k$. Choose a closed embedding $i\colon U\to\A^{m}_k$. The normal bundle $N_i\cong (T_{\A^{m}_k}|_{U})/T_{U}$ is stably trivial. Composing $i$ with the closed embedding $\A^{m}_k\to \A^{m+n}_k$, $t\mapsto (t,0)$, for $n$ big enough we may assume that the tangent bundle $N_i$ is trivial and $U$ is closed in $\A^m_k$. Thus $\Frn_0(\A^m_k\xleftarrow{i} U)\neq\emptyset$. Applying Lemma~\ref{lem:frameeqsur} we get the claim.
\end{proof}

\begin{dfn}\label{cf}
	Let $U\subset \A^{m}_X$, $a\in \Fr_0(U\xleftarrow{i} Y)$ be a framed enhancement of a relative dimension $d=\dim Y-\dim X$ morphism $f\colon Y\to X$. Then for every closed subset $S\subset Y$ finite over $X$ the morphism
	\[
	a\colon U/(U-i(Y)) \to Y_+\wedge T^{m-d}
	\]	
	gives rise to a morphism of sheaves
	\[
	a_f \colon X/(X-f(S))\wedge \PP^{\w m} \to X_+\wedge \PP^{\w m}/	(X_+\wedge \PP^{\w m}-i(S)) \xleftarrow{\simeq} U/(U-i(S)) \to Y/(Y-S)\wedge T^{m-d}.
	\]
\end{dfn}

\begin{lem}\label{pfcoincide}
	Let $M_*$ be a homotopy module. Then in the notation of Definition~\ref{cf} the following diagram commutes.
	\[
	\xymatrix{
	H^i_S(Y,M_*\otimes\omega_{Y/X}) \ar[rr]^{f_*} \ar[dr]_{\otimes \det a} & &  H^{i-d}_{f(S)}(X,M_{*-d})\\
	& H^i_S(Y,M_*) \ar[ur]_{a_f^*} & 
}
\]
Here $\det a$ is given by Definition~\ref{dfn:det} and $a_f^*$ is induced by the pullback along $a_f$ similar to Definition~\ref{frcohaction}.
\end{lem}
\begin{proof}
For every closed subset $Z\subset S$ the morphism $a_f$ induces a morphism 
\[
a_f^Z\colon  X/(X-f(Z))\w\PP^{\w m}\to Y/(Y-Z)\w T^{m-d}
\]
giving a homomorphism 
\[
(a_f^Z)^*\colon H^i_{Z}(Y,M_*)\to H^{i-d}_{f(Z)}(X,M_{*-d}).
\]
Taking the respective colimit we obtain a morphism of Cousin complexes
\[
a_f^*\colon C^\bullet_S(Y,M_*)\to C^\bullet_{f(S)}(X,M_{*-d})[-d]
\]
inducing the map $a_f^*$ from the statement of the lemma.

It is sufficient to check that for every $y\in Y^{(n)}\cap S$ the following diagram commutes.
\[
	\xymatrix{
	H^n_y(Y,M_*\otimes\omega_{Y/X}) \ar[rr]^{f_*} \ar[dr]_{\otimes \det a} & &  H^{n-d}_{f(y)}(X,M_{*-d})\\
	& H^n_y(Y,M_*) \ar[ur]_{(a_f^y)^*} & 
}
\]
For a point $y\in Y^{(n)}\cap S$ choose a Zariski open subset $y\in U\subset Y$ such that $V=\overline{y}\cap U$ is smooth and the closed embedding $i\colon V\to U$ admits a framed enhancement $c\in\Fr_0(U\xleftarrow{i} V)$. For $W=f(U)$ the correspondence $a_f$ induces morphism 
\[
a_f^V\colon W/(W-f(V))\w\PP^{\w m}\to U/(U-V)\w T^{m-d}.
\] 
Then $c\circ a_f^V\in \Fr_{m}(W\xleftarrow{f\circ i} V)$
is a framed enhancement of the finite map $f\circ i\colon V\to W$. 
Consider the following diagram.
\[
\xymatrix{
H^0(V,M_{*-n}\otimes\omega_{V/X}) \ar[dd]_{\otimes \det (c\circ a_f^V)}^{\cong} \ar[rr]^(0.55){(f\circ i)_*} \ar[dr]^{i_*}_{\cong} & & H^{n-d}_{f(V)}(W,M_{*-d}) \\
& H^n_V(U,M_*\otimes\omega_{U/X})\ar[ru]^{f_*} \ar[dr]^{\otimes \det a}  & \\
H^0(V,M_{*-n}) \ar[rr]^(0.55){c^*} & &  H^n_V(U,M_*) \ar[uu]_{(a_f^V)^*}
}
\]
The outer contour and the lower square commute by the definition of pushforward maps for finite morphism. The upper triangle commutes by functoriality of pushforward maps. The map $i_*$ is an isomorphism by Lemma~\ref{lem:trsupiso}. Thus the right triangle commutes as well and taking colimits along neighborhoods of $y$ we obtain the claim.
\end{proof}

\begin{lem}\label{basechange2}
	Let $f\colon Y\to X$  be a morphism in $\Sm_k$ of relative dimension $d=\dim Y-\dim X$. Suppose that $f$ admits a framed enhancement. Then for $c\in \Fr_n(V,W)$ and a closed subset $S\subset W\times Y$ finite over $W\times X$ the following diagram commutes.
	\[
	\xymatrix{
		H^i_S(W\times Y,M_*\otimes\omega_{Y/X})\ar[d]^{(\id_W\times f)_*}\ar[r]^(0.47){(c\boxtimes \id_X)^*} & H^i_{c^{-1}(S)}(V\times Y,M_*\otimes \omega_{Y/X})\ar[d]^{(\id_V\times f)_*}\\
		H^{i-d}_{f(S)}(W\times X,M_{*-d})\ar[r]^(0.47){(c\boxtimes \id_Y)^*} & H^{i-d}_{(f\times c^{-1})(S)}(V\times X,M_{*-d})
	}\]
	Here were shortened the notation as
	\[
	c^{-1}(S)=(c\boxtimes \id_X)^{-1}(S),\quad f(S)=(\id_W\times f)(S),\quad (f\times c^{-1})(S)= (\id_V\times f) ((c\boxtimes \id_X)^{-1}(S)).
	\]
\end{lem}
\begin{proof}
	Let $U\subset \A^{m}_X$, $a\in \Fr_0(U\xleftarrow{i} Y)$ be a framed enhancement of $f$. Lemma~\ref{pfcoincide} yields that 
	\[
	(\id_W\times f)_*=(a_{\id_W\times f})^*\circ (\otimes \det a),\quad (\id_V\times f)_*=(a_{\id_V\times f})^*\circ (\otimes \det a).
	\]
	The claim follows from the commutativity of the following diagram which is straightforward.
	\[
	\xymatrix @C=5em{
		\frac{V\times X}{V\times X -(f\times c^{-1})(S)}\w\PP^{\w m}\w\PP^{\w n}\ar[r]^{(c\boxtimes \id_X)\w \id}\ar[d]_{a_{(\id_V\times f)}\w\id} &  \frac{W\times X}{W\times X - f(S)}\w \PP^{\w m}\w T^n\ar[d]^{a_{(id_W\times f)}\w \id}\\
		\frac{V\times Y}{V\times Y-c^{-1}(S)}\w T^{m-d}\w\PP^{\w n}\ar[r]^{(c\boxtimes \id_Y)\w \id} & \frac{W\times Y}{W\times Y-S}\w T^{m-d}\w T^n
	}
	\]
\end{proof}

\begin{lem}\label{crossprodcommute} Let $f\colon Y\to X$  be a morphism in $\Sm_k$ of relative dimension $d=\dim Y-\dim X$ and $S\subset Y$ be a closed subset finite over $X$. Suppose that $f$ admits a framed enhancement. Then for a homotopy module $M_*$ and $c\in H^j_{S'}(V,\KMW_n)$ the following diagram commutes.
	\[
	\xymatrix{
		H^i_S(Y,M_*\otimes\omega_{Y/X})\ar[r]^-{f_*}\ar[d]^{-\times c} & H^{i-d}_{f(S)}(X,M_{*-d})\ar[d]^{-\times c}\\
		H^{i+j}_{S\times S'}(Y\times V,M_{*+n}\otimes\omega_{Y/X})\ar[r]^-{(f\times \id_V)_*} & H^{i-d+j}_{f(S)\times S'}(X\times V,M_{*+n-d})
	}
	\]
\end{lem}
\begin{proof}
	Let $U\subset \A^{m}_X$, $a\in \Fr_0(U\xleftarrow{i} Y)$ be a framed enhancement of $f$. Lemma~\ref{pfcoincide} yields that 
	\[
	f_*=a_{f}^*\circ (\otimes \det a),\quad (f\times \id_V)_*=a_{f\times \id_V}^*\circ (\otimes \det a).
	\]
	Moreover, $a_{f\times \id_V}=a_f\boxtimes \id_{V/V-S'}$.

	For $m\in H^i_S(Y,M_*)$ given by a morphism 
	\[
	m\colon \Z[Y/Y-S] \to M_*[i]
	\]
	in $\mathrm{D}(\Ab(k))$ the external product $m\times c$ is given by the map 
	\[
	m\times c\colon \Z[Y/Y-S]\otimes\Z[V/V-S']\to M_*[i]\times \KMW_n[j]\to M_{*+n}[i+j].
	\]
	Then 
	\[
	(f\times \id_V)_*(m\times c)=a_{f\times \id_V}\circ (m\times c)=(a_f\circ m)\times c = f_*(m)\times c
	\]
	and the claim follows.
\end{proof}

\section{Milnor-Witt correspondences}\label{mwc}

\begin{dfn}\cite[\S 4]{CF} For $X,Y\in\Sm_k$ with $\dim Y=d$ and a closed subset $S\subset X\times Y$ put
\[
\wCor_S(X,Y)=H^d_S(X\times Y,\KMW_d\otimes\omega_Y).
\]		
The group of finite MW-correspondences is defined as
\[
\wCor(X,Y)=\colim\limits_S \wCor_S(X,Y)
\]
where the colimit is taken along all the admissible closed subsets $S\subset X\times Y$, i.e. the closed subsets $S\subset X\times Y$ such that every irreducible component of $S$ is finite and surjective over some irreducible component of $X$. See~\cite[\S 4.2]{CF} for the definition of composition of finite MW-correspondences leading to the category $\wCor_k$ with the objects those of $\Sm_k$ and the morphisms given by $\wCor(X,Y)$. An additive contravariant functor $\wCor_k\to \mathrm{Ab}$ is called a presheaf with MW-transfers.
\end{dfn}

\begin{rem} \label{rem:untwistcor}
For $X,Y\in\Sm_k$ with $\dim Y=d$ and a closed subset $S\subset X\times Y$ let $L_Y$ be the total space of the canonical line bundle $\omega_Y$ and $i\colon X\times Y\to X\times L_Y$ be the zero section. Then
\begin{align*}\label{wcorascoh}
i_*\colon\wCor_S(X,Y)=H^d_S(X\times Y,\KMW_d\otimes\omega_Y)\xrightarrow{\simeq} H^{d+1}_S(X\times L_Y,\KMW_{d+1}).
\end{align*}
is an isomorphism by Corollary~\ref{killtwist}.
\end{rem}

\begin{dfn}\label{frcorpairing} Let $X,Y,V\in\Sm_k$ and $S\subset X\times Y$ be a closed subset. Then every $a\in\Fr_n(V,X)$ defines a unique homomorphism $a^*$ making the following diagram commute.
\[
\xymatrix @C=7em{
\wCor_S(X,Y) \ar[r]^{a^*} \ar[d]_{\cong} & \wCor_{a^{-1}(S)}(V,Y) \ar[d]^\cong\\
H^{d+1}_S(X\times L_Y,\KMW_{d+1}) \ar[r]^{(a\boxtimes \id_{L_Y})^*} & H^{d+1}_{a^{-1}(S)}(V\times L_Y,\KMW_{d+1})
}
\]
Here $a^{-1}(S)=(a\boxtimes\id_Y)^{-1}(S)$, the vertical isomorphisms are given by Remark~\ref{rem:untwistcor} and the bottom morphism is the pullback on the cohomology introduced in Definition~\ref{frcohaction}.

Note that if $S\subset X\times Y$ is admissible then $a^{-1}(S)\subset V\times Y$ is admissible as well. The homomorphism $a^*$ commutes with the inclusions of admissible subsets and gives rise to a homomorphism
\[
a^*\colon\wCor(X,Y)\to\wCor(V,Y).
\]
It is straightforward to check that $(a\circ b)^*=b^*\circ a^*$ for $W\in\Sm_k$ and $b\in \Fr_m(W,V)$. Then the above rule endows $\wCor(-,Y)$ with the structure of a framed presheaf. Lemma~\ref{lem:additivity} yields that this structure descends to the structure of a $\ZF_*$-presheaf and gives rise to a homomorphism
\[
\mathcal{D}\colon \ZF_*(X,Y)\to\wCor(X,Y),\quad \mathcal{D}(a)=a^*(1_Y)
\]
where $1_Y\in\wCor(Y,Y)$ is the identity morphism.
\end{dfn}

\begin{lem}\label{nochoice}
	In the notation of Definition~\ref{frcorpairing} suppose that $\omega_Y$ is trivial and fix a trivialization $\theta\colon \omega_Y\xrightarrow{\simeq} \Os_Y$. Then the following diagram commutes.
	\[
	\xymatrix @C=7em{
		\wCor_S(X,Y) \ar[r]^{a^*} \ar[d]_{=} & \wCor_{a^{-1}(S)}(V,Y) \ar[d]^=\\
		H^{d}_S(X\times Y,\KMW_{d}\otimes \omega_Y) \ar[d]_{\cong}^{\otimes \theta} & H^{d}_{a^{-1}(S)}(V\times Y,\KMW_{d}\otimes \omega_Y) \ar[d]^{\cong}_{\otimes \theta} \\
		H^{d}_S(X\times Y,\KMW_{d})\ar[r]^{(a\otimes \id_Y)^*} & H^{d}_{a^{-1}(S)}(V\times Y,\KMW_{d})
	}
	\]
\end{lem}
\begin{proof}
Follows from Corollary~\ref{killtwist}.
\end{proof}

\begin{dfn}\label{coractiondef}
Let $X,Y,V\in\Sm_k$, $S\subset V\times X$ be an admissible closed subset and $Z\subset X\times Y$ be a closed subset. Denote $p\colon V\times X\times Y\to V\times Y$ the projection and put $d=\dim X$. For a homotopy module $M_*$ define a pairing
\[
\wCor_S(V,X)\times H^i_Z(X\times Y,M_*)\xrightarrow{\cup} H^i_{S\cdot Z}(V\times Y,M_*)
\]
with $S\cdot Z=p(S\times Y\cap V\times Z)$ as the following composition
\[
\xymatrix{
	\wCor_S(V,X)\times H^i_Z(X\times Y,M_*) \ar[d]^= \ar[r]^\cup &  H^i_{S\cdot Z}(V\times Y,M_*)\\
H^d_S(V\times X,\KMW_d\otimes\omega_X)\times H^i_Z(X\times Y,M_*) \ar[dr]_{\times} & H^{d+i}_{(S\times Y)\cap (V\times Z)}(V\times X\times Y,M_{*+d}\otimes\omega_X) \ar[u]_{p_*} \\
& H^{d+i}_{S\times Z}(V\times X\times X\times  Y,M_{*+d}\otimes\omega_X) \ar[u]_{\Delta_X^*}
}
\]
\end{dfn}

\begin{lem}\label{coractionrespectsfr}
In the notation of Definition~\ref{coractiondef} suppose that $X\to\Spec k$ admits a framed enhancement. Then for every $W\in \Sm_k$ and $a\in\Fr_n(W,V)$ the following square commutes.
\[
\xymatrix{
\wCor_S(V,X)\times H^i_Z(X\times Y,M_*)\ar[r]^(0.58)\cup \ar[d]^{a^*\times \id} & H^i_{S\cdot Z}(V\times Y,M_*)\ar[d]^{(a\boxtimes \id_Y)^*}\\
\wCor_{a^{-1}(S)}(W,X)\times H^i_Z(X\times Y,M_*)\ar[r]^(0.58)\cup & H^i_{a^{-1}(S)\cdot Z}(W\times Y,M_*)
}
\]
\end{lem}
\begin{proof}
Recall that
	\[
	\wCor_S(V,X)=H^d_S(V\times X,\KMW_d\otimes\omega_X),\quad \wCor_{a^{-1}(S)}(W,X)=H^d_{a^{-1}(S)}(W\times X,\KMW_d\otimes\omega_X).
	\]
Then for every $b\in H^i_Z(X\times Y,M_*)$ we need to prove  commutativity of the following diagram.
\[
\xymatrix @C=3em{
	H^d_S(V\times X,\KMW_d\otimes\omega_X)\ar[r]^(0.48){(a\boxtimes \id)^*}\ar[d]^-{-\times b} 
	& 
	H^d_{a^{-1}(S)}(W\times X,\KMW_d\otimes\omega_X)\ar[d]^-{-\times b}
	\\
	H^{d+i}_{S\times Z}(V\times X\times X\times  Y,M_{*+d}\otimes\omega_X)\ar[r]^(0.48){(a\boxtimes \id)^*}\ar[d]^-{\Delta_X^*} 
	& 
	H^{d+i}_{a^{-1}(S)\times Z}(W\times X\times X\times Y,M_{*+d}\otimes\omega_X)\ar[d]^{\Delta_X^*} \\
	H^{d+i}_{(S\times Y)\cap (V\times Z)}(V\times X\times Y,M_{*+d}\otimes\omega_X)\ar[r]^(0.48){(a\boxtimes \id)^*}\ar[d]^-{p_{*}} 
	& 
	H^{d+i}_{(a^{-1}(S)\times Y)\cap (W\times Z)}(W\times X\times Y,M_{*+d}\otimes\omega_X)\ar[d]^-{q_*} \\
	H^i_{S\cdot Z}(A\times Y,M_*)\ar[r]^(0.48){(a\boxtimes \id)^*}
	& H^i_{a^{-1}(S\cdot Z)}(W\times Y,M_*)
}
\]
Here $p\colon V\times X\times Y\to V\times Y$ and $q\colon W\times X\times Y\to W\times Y$ are the projections. The top square commutes by Lemma~\ref{frcomextprod}, the middle square commutes since it consists of pullbacks, the bottom square commutes by Lemma~\ref{basechange2}.
\end{proof}

\begin{lem}\label{ourcorcomp}
	In the notation of Definition~\ref{coractiondef} put $d'=\dim Y$, let $L_Y$ be the total space of the canonical line bundle $\omega_Y$ and let $S'\subset X\times Y$ be an admissible closed subset. Then the following diagram commutes.
\[
\xymatrix{
\wCor_S(V,X)\times H^{d'+1}_{S'}(X\times L_Y,\KMW_{d'+1})\ar[r]^(0.6)\cup & H^{d'+1}_{S\cdot S'}(V\times L_Y,\KMW_{d'+1}) \\
\wCor_S(V,X)\times\wCor_{S'}(X,Y)\ar[u]_{\id\times i_*}^{\cong} \ar[r]^(0.6)\circ  & \wCor_{S\cdot S'}(V,Y) \ar[u]_{(\id\times i)_*}^{\cong}
}
\]
Here $i_*$ is the isomorphism from Remark~\ref{rem:untwistcor} and $\circ$ is the composition of MW-correspondences introduced in~\cite[\S 4.2]{CF}
\end{lem}
\begin{proof}
Recall that \cite[Lemma~4.6]{CF} yields that for an open embedding $V_0\subset V$ the restriction 
\[
\wCor_{S\cdot S'}(V,Y)\to  \wCor_{(S\cdot S')\cap (V_0\times Y)}(V_0,Y)
\]
is injective, thus it is sufficient to prove the claim for an open subset $V_0$. Note that 
\[
\wCor_{S\cap (V_0\times X)}(V_0,X)= \wCor_{S\cap (V_0\times X_0)}(V_0,X_0)
\]
for every open subset $X_0\subset X$ such that $\pi(S\cap V_0\times X) \subset X_0$ for the projection $\pi\colon V\times X\to X$. Similarly
\begin{gather*}
\wCor_{S'\cap (X_0\times Y)}(X_0,Y)= \wCor_{S'\cap (X_0\times Y_0)}(X_0,Y_0),\\
\wCor_{(S\cdot S')\cap (V_0\times Y)}(V_0,Y)= \wCor_{(S\cdot S')\cap (V_0\times Y_0)}(V_0,Y_0).
\end{gather*}
Thus taking $V_0$ to be sufficiently small, shrinking $X$ to $X_0$ and $Y$ to $Y_0$ accordingly and applying Lemma~\ref{semilocetale} we may assume that $X$ and $Y$ admit \'etale maps to affine spaces.

Let $\theta\colon \omega_Y\xrightarrow{\simeq} \Os_Y$ be the trivialization of $\omega_Y$ given by an \'etale map $Y\to \A^{d'}_k$. Consider the following diagram.
\[
\xymatrix @C=7em{
\wCor_S(V,X)\times H^{d'+1}_{S'}(X\times L_Y,\KMW_{d'+1})\ar[r]^(0.6)\cup & H^{d'+1}_{S\cdot S'}(V\times L_Y,\KMW_{d'+1}) \\
\wCor_S(V,X)\times H^{d'}_{S'}(X,\KMW_{d'}\otimes \omega_Y) \ar[u]_{\id\times i_*}^{\cong} \ar[d]^{\id \times (\otimes \theta)}_{\cong} & \wCor_{S\cdot S'}(V,Y) \ar[u]_{(\id\times i)_*}^{\cong}\\
\wCor_S(V,X)\times H^{d'}_{S'}(X,\KMW_{d'}) \ar[ru]^(0.6)\cup & \wCor_S(V,X)\times\wCor_{S'}(X,Y) \ar[u]^\circ  \ar[l]_(0.45){\id \times (\otimes \theta)}^(0.45){\cong}
}
\]
The bottom triangle commutes: compare the formula from Definition~\ref{coractiondef} with \cite[\S 4.2]{CF} in the case of $\omega_Y\cong \Os_Y$ (see also Remark~\ref{coincideMorel}). The commutativity of the upper half is straightforward, since $i_*$ commutes with all the involved pullbacks (Lemma~\ref{basechange2}), pushforwards (functoriality of pushforward maps) and external products (Lemma~\ref{crossprodcommute}).
\end{proof}

\begin{prop}\label{prop:functorcom} The map $\mathcal{D} \colon \ZF_*(X,Y)\to\wCor(X,Y)$ from Definition~\ref{frcorpairing} gives rise to a functor 
\[
\mathcal{D}\colon \ZF_*(k)\to \wCor_k.
\]
\end{prop}
\begin{proof}
	Let $X,Y,V,W\in \Sm_k$ and $S_1\subset W\times X,\, S_2\subset X\times Y$ be admissible closed subsets. For $a\in\Fr_n(V,W)$ consider the following diagram.
	\[
	\xymatrix{
		\wCor_{S_1}(W,X)\times\wCor_{S_2}(X,Y)\ar[r]^(0.6){\circ}\ar[d]^{a^*\times \id} & \wCor_{S_1\cdot S_2}(W,Y)\ar[d]^{a^*}\\
		\wCor_{a^{-1}(S_1)}(V,X)\times\wCor_{S_2}(X,Y)\ar[r]\ar[r]^(0.6){\circ} & \wCor_{a^{-1}(S_1)\cdot S_2}(V,Y)
	}
	\]
Applying the reasoning similar to the one used in the beginning of the proof of Lemma~\ref{ourcorcomp} we may assume that $X\to \Spec k$ admits a framed enhancement. Then the diagram commutes by Lemmas~\ref{coractionrespectsfr} and~\ref{ourcorcomp}. 
	
	Then for $b\in\ZF_*(Y,Z)$ and $a\in\ZF_*(X,Y)$ we have 
	\[
	\mathcal{D}(b\circ a)=(b\circ a)^*(1_Z)=a^*(b^*(1_Z))=a^*((b^*(1_Z))\circ 1_Y)=b^*(1_Z)\circ a^*(1_Y)=\mathcal{D}(b)\circ \mathcal{D}(a),
	\]
Since $\mathcal{D}$ clearly respects the identity morphisms the claim follows.
\end{proof}

\section{Milnor-Witt transfers on homotopy modules}\label{mwohm}

\begin{dfn}\label{corpairing} Let $X,Y\in\Sm_k$ and $M_*$ be a homotopy module. Then for every admissible closed subset $S\subset Y\times X$ Definition~\ref{coractiondef} gives a pairing
\[
\wCor_S(Y,X)\times M_*(X)\xrightarrow{\cup} M_*(Y)
\]
inducing a pairing
\[
\wCor(Y,X)\times M_*(X)\xrightarrow{\cup} M_*(Y).
\]
\end{dfn}

\begin{lem}\label{cortozfopen}
	In the notation of Definition~\ref{corpairing} let $j\colon Y_0\to Y$ be an open embedding. Then the following diagram commutes.
	\[
	\xymatrix{
		\wCor (Y,X)\times M_*(X)\ar[d]_{j^*\times \id}\ar[r]^(0.65)\cup & M_*(Y)\ar[d]^{j^*}\\
		\wCor (Y_0,X)\times M_*(X)\ar[r]^(0.65)\cup & M_*(Y_0)	
	}
	\]
\end{lem}
\begin{proof}
This follows from the fact that pullback along an open embedding commutes with external products (Lemma~\ref{frcomextprod}), pullbacks (functoriality of pullbacks) and pushforward maps (Lemma~\ref{lem:opencommute}).
\end{proof}

\begin{lem}\label{cortozfonfr}In the notation of Definition~\ref{corpairing} for $V\in \Sm_k$ and $a\in \Fr_n(V,Y)$ the following diagram commutes.
\[
\xymatrix{
\wCor(Y,X)\times M_*(X)\ar[d]_{a^*\times \id}\ar[r]^(0.65){\cup} & M_*(Y)\ar[d]^{a^*}\\
\wCor(V,X)\times M_*(X)\ar[r]^(0.65){\cup} & M_*(V)
}
\]
\end{lem}
\begin{proof}
Recall that for every open subset $V_0\subset V$ the restriction $M_*(V)\to M_*(V_0)$ is injective. Hence by Lemma~\ref{cortozfopen} it is sufficient to prove the statement for any open subset $V_0$ of $V$. Applying the reasoning similar to the one used in the beginning of the proof of Lemma~\ref{ourcorcomp} and shrinking $V,X$ and $Y$ we may assume that $X\to \Spec k$ admits a framed enhancement. Then the diagram commutes by Lemma~\ref{coractionrespectsfr}.
\end{proof}


\begin{lem}\label{idok} For $X\in\Sm_k$, a homotopy module $M_*$ and the identity morphism $1_X\in\wCor(X,X)$ the following holds:
	\[
	(1_X\cup -) =\id \colon M_*(X)\to M_*(X).
	\]
\end{lem}
\begin{proof}
	Recall that for every open subset $X_0\subset X$ the restriction $M_*(X)\to M_*(X_0)$ is injective, then, applying Lemmas~\ref{cortozfopen} and~\ref{semilocetale} we may assume that $X$ admits an \'etale map $\pi\colon X \to \A^d_k$. Consider the map 
	\[\phi\colon X\times X\to\A^d_k, \ \ \phi(x_1,x_2)=\pi(x_1)-\pi(x_2).\] 
Since $\pi$ is \'etale, the zero set $Z(\phi)$ is $Z(\phi)=\Delta(X)\sqcup \mathcal{Z}'$ for some closed $\mathcal{Z}'$.	
Then
	\[
	a=(\Delta(X),U,\phi, p_2)\in \Fr_0(X\times X \xleftarrow{\Delta} X)
	\]
	with $U=X\times X-\mathcal{Z}'$ and $p_1,p_2\colon X\times X\to X$ being the projections is a framed enhancement of the diagonal morphism $X\xrightarrow{\Delta} X\times X$. Fix the trivialization of $\omega_X$ induced by $a$.
	
	Recall that $1_X=\Delta_*(1)$, thus we need to check that for every $m\in M_*(X)$ one has
	\[
	p_{1*}(\id_X\times \Delta)^*(\Delta_*(1)\times m) = m
	\]
	in the following diagram.
	\[
	\xymatrix @C=6em{
	H^d_{\Delta(X)}(X\times X, \KMW_d\otimes p_2^*\omega_X) \times M_*(X) \ar[r]^(0.55)\cup  \ar[d]_\times & M_*(X)\\
	H^d_{\Delta(X)\times X}(X\times X\times X, M_{*+d}\otimes p_2^*\omega_X) \ar[r]^(0.55){(\id_X\times \Delta)^*} &
	H^d_{\Delta(X)}(X\times X, M_{*+d}\otimes p_2^*\omega_X) \ar[u]^{p_{1*}}
	}
	\]
	Lemma~\ref{pfcoincide} yields that $\Delta_*(1)=a^*_{\Delta}(1)$ whence 
	\begin{gather*}
	(\id_X\times \Delta)^*(\Delta_*(1)\times m) =(\id_X\times \Delta)^* (a^*_\Delta(1)\times m)= a^*_\Delta(m)=\Delta_*(m),\\
	p_{1*}(\id_X\times \Delta)^*(\Delta_*(1)\times m)=p_{1*}\Delta_*(m)=m.\qedhere
	\end{gather*}
%
\end{proof}

\begin{lem}\label{cortozfoncor}
	In the notation of Definition~\ref{corpairing} for $V\in \Sm_k$ and $\alpha \in\wCor(V,Y)$ the following diagram commutes.
\[
\xymatrix{
\wCor(Y,X)\times M_*(X)\ar[d]_{(\circ\alpha) \times \id}\ar[r]^(0.65)\cup & M_*(Y)\ar[d]^{\alpha\cup }\\
\wCor(V,X)\times M_*(X)\ar[r]^(0.65)\cup & M_*(V)
}
\]
\end{lem}
\begin{proof}
	Applying the reasoning similar to the one in the beginning of the proof of Lemma~\ref{idok} we may assume that $X$ and $Y$ admit \'etale morphisms to the affine spaces of respective dimension. Lemma~\ref{ourcorcomp} yields that $\circ$ can be interpreted in the terms of $\cup$ which is defined using external product, pullback and pushforward. A straightforward although rather lengthy computation shows that the claim follows from Lemmas~\ref{basechange2} and~\ref{crossprodcommute}.
\end{proof}

\begin{dfn} \label{dfn:phimw}
	Let $\HMCor$ denote the category of homotopy modules with MW-transfers, i.e. the category of pairs $(M_*,\phi_*)$ where $M_*$ is a $\Z$-graded homotopy invariant sheaf with MW-transfers and $\phi_i\colon M_i\to (M_{i+1})_{-1}$ are isomorphisms of sheaves with MW-transfers. We usually shorten the notation and refer to $(M_*,\phi_*)$ as $M_*$. Definition~\ref{frcorpairing} (see also Proposition~\ref{prop:functorcom}) gives rise to a functor 
	\[
	\Phi^{\MW}\colon \ShvMW \to \Shvfr
	\] 
	which descends to a functor
	\[
	\Phi^{\MW}\colon \HMCor \to \HMfr.
	\]
	Here $\ShvMW$ and $\Shvfr$ stand for the categories of sheaves with MW-transfers and $\ZF_*$-sheaves respectively. Definition~\ref{corpairing} (see also Lemmas~\ref{idok} and~\ref{cortozfoncor}) gives rise to a functor
	\[
	\Psi^{\MW}\colon \HM \to \ShvMW.
	\]
	The functor $\Psi^{\MW}$ takes an isomorphism of graded sheaves $\phi_*\colon M_*\xrightarrow{\simeq} (M_{*+1})_{-1}$ to an isomorphism  of graded sheaves with MW-transfers, thus it descends to a functor
	\[
	\Psi^{\MW}\colon \HM \to \HMCor.
	\]
\end{dfn}

\begin{dfn}
For $X\in\Sm_k$ denote $\HCor(X)$ the sheaf associated with the presheaf 
	\[
	U\mapsto \coker\left[\wCor(U\times\A^1,X)\xrightarrow{i_0^*-i_1^*}\wCor(U,X)\right].
	\]
	Here $i_0,i_1\colon U\to U\times \A^1$ are given by $i_0(u)=(u,0)$ and $i_1(u)=(u,1)$ respectively. Note that $\HCor(X)$ carries a canonical structure of a presheaf with MW-transfers.
\end{dfn}

\begin{thm}\label{hcor=hzf}
	 For $X\in \Sm_k$ the functor $\mathcal{D}\colon \ZF_*(k)\to\wCor_k$ from Proposition~\ref{prop:functorcom} induces an isomorphism of sheaves
\[
f\colon  \HZF(X)\xrightarrow{\simeq}\HCor(X).
\]

\end{thm}
\begin{proof}
	Let $\overline{\id}_X\in\HZF(X)(X)$ be the image of the identity $\id_X\in\ZF(X)(X)$. Then Definition~\ref{corpairing} yields a morphism
	\[
	\wCor(-,X)\xrightarrow{\cup \overline{\id}_X} (\HZF(X)_*)_0.
	\]
	Here $\HZF(X)_*$ is the framed homotopy module given by Definition~\ref{dfn:hzf}. Since the right-hand side is homotopy invariant, it descends to a morphism
	\[
	g\colon \HCor(X)\to \HZF(X)_0=\HZF(X).
	\]
	We need to check that $f$ and $g$ are inverse to each other. For $U\in\Sm_k$, $\overline{a}\in\HZF(X)(U)$ being the image of $a\in\ZF(X)(U)$ Lemmas~\ref{cortozfonfr} and~\ref{idok} yield
	\[
	(g_U \circ f_U) (\overline{a})= a^*(1_X)\cup \overline{\id}_X=a^*(1_X\cup \overline{\id}_X)=a^*(\overline{\id}_X)=\overline{a}
	\]
	whence $g \circ f=\id$. For $\alpha\in \HCor(X)(U)$ we have 
	\[
	(f_U\circ g_U) (\alpha)= (\alpha \cup \overline{\id}_X)^*(1_X)=(f_U(\alpha \cup \overline{\id}_X))\circ 1_X = \alpha \circ f_U(\overline{\id}_X) =\alpha.
	\]
	Here the second equality follows from the fact that the framed structure induced by $\mathcal{D}$ on $\HCor(X)$ agrees with the structure of a presheaf with MW-transfers. The third equality follows from the fact that $f$ can be promoted to a morphism of presheaves with MW-transfers with $\HZF(X)$ being endowed with the structure of a presheaf with MW-transfers applying $\Psi^{\MW}$ to $\HZF(X)_*$. Thus $f\circ g=\id$ and the claim follows.
\end{proof}

\begin{cor} For $X\in\Sm_k$ there is a canonical isomorphism of sheaves
\[
\pi_0(\Sigma_{\PP^1}^\infty X_+)_0\cong \HCor(X).
\]
\end{cor}
\begin{proof}
	Follows from Theorem~\ref{hcor=hzf} combined with the isomorphism
$\pi_0(\Sigma_{\PP^1}^\infty X_+)_0\cong \HZF(X)$ given by~\cite[Theorem~11.1 and Corollary~11.3]{GP}.
\end{proof}

\begin{rem}
The above corollary provides a reasonable description for the zeroth stable motivic homotopy sheaf of $X$. A similar computation with $X$ being a smooth projective variety was obtained by different means in \cite[Theorem~4.3.1]{AH} and \cite[Theorem~1.2]{Ananyevskiy}, see also \cite[Theorem~1.3]{Ananyevskiy} for the case of a smooth curve.
\end{rem}

\begin{cor}\label{cor:addmw}
Let $M$ be a homotopy invariant stable $\ZF_*$-sheaf. Then $M$ admits a unique structure of a presheaf with MW-transfers compatible with the framed structure.
\end{cor}
\begin{proof}
	To give $M$ a structure of a presheaf with MW-transfers compatible with the framed structure is the same thing as to construct for every $X\in\Sm_k$ a map $\phi_X$ such that the following diagram commutes
	\[
	\xymatrix{
	\ZF_*(-,X)\times M(X) \ar[r] \ar[d]_{\mathcal{D} \times \id} & M \\
	\wCor(-,X)\times M(X) \ar[ur]_(0.6){\phi_X} & 
}
	\]
	and check associativity relation. Here the horizontal arrow is given by the framed structure and the vertical one is induced by the functor $\mathcal{D}$. Since $M$ is homotopy invariant and stable then it is the same as to construct $\overline{\phi}_X$ for the following diagram.
	\[
\xymatrix{
	\HZF(X)\times M(X) \ar[r] \ar[d]_{\mathcal{D} \times \id} & M \\
	\HCor(X)\times M(X) \ar[ur]_(0.6){\overline{\phi}_X} & 
}
\]
Here the vertical morphism is an isomorphism by Theorem~\ref{hcor=hzf} so the claim follows.

\end{proof}

\begin{thm}\label{hmodeq}
The forgetful functors $\Phifr$ and $\Phi^\MW$ from Definitions~\ref{dfn:phifr} and~\ref{dfn:phimw} respectively establish equivalences of categories 
\[
\HMCor\underset{\simeq}{\xrightarrow{\Phi^\MW}}\HMfr\underset{\simeq}{\xrightarrow{\Phifr}}\HM.
\]
\end{thm}
\begin{proof}
The functor $\Phifr$ is an equivalence by Proposition~\ref{unfrstr}. Lemma~\ref{cortozfonfr} yields that 
\[
\Phi^\MW\circ \Psi^\MW \circ \Phifr =\id.
\]
Then it is sufficient to check that if $\Phi^\MW\circ \Phifr(M_*)\cong \Phi^\MW\circ \Phifr(M'_*)$ then $M_*\cong M'_*$, i.e. that a framed homotopy module carries a unique structure of a homotopy module with MW-transfers compatible with the framed structure. The reasoning is the same as in the proof of Corollary~\ref{cor:addmw}. The structure of a presheaf with MW-transfers is given by morphisms of presheaves $\wCor(-,X)\times M_*(X)\to M_*$ which descend to
\[
\HCor(X)\times M_*(X)\to M_*.
\]
The latter one is unique (compatible with the framed structure) since Theorem~\ref{hcor=hzf} yields an isomorphism of sheaves $\HCor(X)\cong \HZF(X)$ and the morphism
\[
\HCor(X)\times M_*(X)\cong \HZF(X)\times M_*(X)\to M_*
\]
is induced by the framed structure $\ZF(X)\times M_*(X)\to M_*$.
\end{proof}

\begin{rem}
Recall that the category of homotopy modules is equivalent to the heart of the homotopy $t$-structure on the stable $\A^1$-derived category $D_{\A^1}(k)$ (see, for example, \cite[Theorem~3.3.3]{AH}). In a similar way one can show that the category of homotopy modules with MW-transfers is equivalent to the heart of the homotopy $t$-structure on the category of MW-motives $\wDM(k)$ introduced in~\cite{DF}.
\end{rem}

\begin{cor} \label{prop:hearts}
The forgetful functor $\wDM(k)\to D_{\A^1}(k)$ induces an equivalence on the hearts of the homotopy $t$-structures.
\end{cor}
\begin{proof}
	Follows from Theorem~\ref{hmodeq}.
\end{proof}

\bibliographystyle{plain}

\begin{thebibliography}{10}
\bibitem{Ananyevskiy} A. Ananyevskiy. On the zeroth stable $\A^1$-homotopy group of a smooth curve. Published online in {\em J. Pure Appl. Algebra}, doi:10.1016/j.jpaa.2017.12.001.


\bibitem{AGP} A. Ananyevskiy, G. Garkusha, I. Panin. Cancellation theorem for framed motives of algebraic varieties. {\em Preprint}, arXiv:1601.06642.

\bibitem{AH} A. Asok, C. Haesemeyer. The 0-th stable $\A^1$-homotopy sheaf and quadratic zero cycles. {\em Preprint}, arXiv:1108.3854.

\bibitem{CF} B. Calm\`es, J. Fasel. The category of finite MW-corrsespondences. {\em Preprint}, arXiv:1412.2989.

\bibitem{CF2} B. Calm\`es, J. Fasel. A comparison theorem for MW-motivic cohomology. {\em Preprint}, arXiv:1708.06100.

\bibitem{CTHK} J.-L. Colliot-Th\'el\`ene, R. Hoobler, B. Kahn. The Bloch-Ogus-Gabber theorem. {\em Algebraic $K$-theory} (Toronto, ON, 1996), 31--94, Fields Inst. Commun., 16, Amer. Math. Soc., Providence, RI, 1997.

\bibitem{DF} F. D\'eglise, J. Fasel. MW-motivic complexes. {\em Preprint}, arXiv:1708.06095 

\bibitem{FO} J. Fasel, P.~\O stvaer. A cancellation theorem for Milnor--Witt correspondences. {\em Preprint}, arXiv:1708.06098.

\bibitem{GP} G. Garkusha, I. Panin. Framed motives of algebraic varieties (after V. Voevodsky) {\em Preprint}, arXiv:1409.4372

\bibitem{GPhi} G. Garkusha, I. Panin. Homotopy invariant presheaves with framed transfers {\em Preprint}, arXiv:1504.00884



\bibitem{Lindel} H. Lindel. On the Bass-Quillen conjecture concerning projective modules over polynomial rings. {\em Invent. Math.} 65 (1981/82), no. 2, 319--323. 

\bibitem{MVW06} C.~Mazza, V.~Voevodsky, C.~Weibel. Lecture notes on motivic cohomology. Clay Mathematics Monographs, vol. 2, American Mathematical Society, Providence, RI, 2006

\bibitem{Morel} F. Morel. $\A^1$-algebraic topology over a field. Lecture Notes in Mathematics, 2052. Springer, Heidelberg, 2012.

\bibitem{Morelintro} F. Morel. An introduction to $\A^1$-homotopy theory. Contemporary developments in algebraic $K$-theory, 357--441.

\bibitem{Neshitov} A. Neshitov. Framed correspondences and the Milnor-Witt K-theory,  J. Inst. Math. Jussieu, available on CJO2016. doi:10.1017/S1474748016000190. 



\bibitem{V} V. Voevodsky. Notes on framed correspondences. Unpublished, 2001, available at math.ias.edu/vladimir/files/framed.pdf

\bibitem{Weibel} C. Weibel. The $K$-book. An introduction to algebraic $K$-theory. Graduate Studies in Mathematics, 145. American Mathematical Society, Providence, RI, 2013.

\end{thebibliography}

\end{document}